\documentclass{amsart}
\usepackage{amssymb}
\usepackage{amsbsy}
\usepackage{amscd}
\usepackage{amsmath}
\usepackage{amsthm}
\usepackage{amsxtra}
\usepackage{latexsym}

\usepackage{mathrsfs}
\usepackage{upgreek}
\usepackage{wasysym}
\usepackage[all,cmtip]{xy}

\usepackage{xcolor}
\definecolor{rouge}{rgb}{0.85,0.1,.4}
\definecolor{bleu}{rgb}{0.1,0.2,0.9}
\definecolor{violet}{rgb}{0.7,0,0.8}

\usepackage[dvipdfmx,colorlinks=true,linkcolor=bleu,urlcolor=violet,citecolor=rouge]{hyperref}

\input{Dynkin.sty}

\newcommand{\bra}{{\langle}}
\newcommand{\ket}{{\rangle}}

\newcommand{\cprime}{$'$}
\newcommand{\on}{\operatorname}
\newcommand{\+}{\mathop{\oplus}}
\renewcommand{\*}{{\otimes}}
\newcommand{\mc}{\mathcal}
\newcommand{\mf}{\mathfrak}

\newcommand{\fing}{\mf{g}}

\newcommand{\affg}{\widehat{\mf{g}}}

\newcommand{\Z}{\mathbb{Z}}
\newcommand{\C}{\mathbb{C}}
\newcommand{\N}{\mathbb{N}}
\newcommand{\Q}{\mathbb{Q}}
\newcommand{\W}{\mathcal{W}}  %{\mathscr{W}}
\newcommand{\ra}{\rightarrow}
\newcommand{\lam}{\lambda}

\newcommand{\bs}{\boldsymbol}

\def\g{\mathfrak{g}}
\def\l{\mathfrak{l}}

\def\h{\mathfrak{h}}
\def\n{\mathfrak{n}}

\def\O{\mathbb{O}}
\def\SS{\mathbb{S}}
\def\Slo{\mathscr{S}}

\def\P{\mathscr{P}}
\def\sl{\mathfrak{sl}}
\def\so{\mathfrak{so}}
\def\sp{\mathfrak{sp}}
\def\eps{\varepsilon}

\def\bs{\boldsymbol}
\def\P{\mathscr{P}}

\def\leq{\leqslant}
\def\geq{\geqslant}

\DeclareMathOperator{\End}{End}

\DeclareMathOperator{\gr}{gr}

\DeclareMathOperator{\ad}{ad}

\DeclareMathOperator{\Sing}{Sing}

\theoremstyle{theorem}
\newtheorem{Th}{Theorem}[section]

\newtheorem{Pro}[Th]{Proposition}

\newtheorem{Lem}[Th]{Lemma}
\newtheorem{lemma}[Th]{Lemma}

\newtheorem{Conj}{Conjecture}
\theoremstyle{remark}
\newtheorem{Def}[Th]{Definition}

\newtheorem{Rem}[Th]{Remark}

\title{On the irreducibility of  associated varieties of W-algebras}
\author{Tomoyuki Arakawa}
\address{Department of Mathematics, MIT, 77 Massachusetts Ave. Cambridge MA 02139 USA}
\address{Research Institute for Mathematical Sciences, Kyoto University,
 Kyoto 606-8502 JAPAN}
 
\email{arakawa@kurims.kyoto-u.ac.jp}
\author{Anne Moreau}
\address{Laboratoire de Math\'{e}matiques et Applications, T\'{e}l\'{e}port 2 - BP 30179, 
Boulevard Marie et Pierre Curie, 86962 Futuroscope Chasseneuil Cedex, France}
\email{anne.moreau@math.univ-poitiers.fr}

\begin{document}

 \begin{abstract}
 We investigate the irreducibility of the nilpotent Slodowy slices that appear as the associated variety of 
 $W$-algebras.
 Furthermore, we provide new examples of vertex algebras whose associated variety has finitely many symplectic leaves.
 \end{abstract}

\maketitle

\begin{center}
{\em Dedicated to the 60th birthday of Professor Efim Zelmanov}
\end{center}

\bigskip 

\section{Introduction}
It is known \cite{Li05} that  every vertex algebra $V$ is canonically filtered 
and therefore it can be considered as a quantization of  its associated graded Poisson vertex algebra
$\gr V$.
The {\em Zhu's $C_2$-algebra} 
 $R_V$ 
of $V$ \cite{Zhu96}
is
a generating subalgebra of the differential algebra $\gr V$ 
and has the structure of a Poisson algebra.
Its maximal spectrum
\begin{align*}
{X}_V:={\rm Specm}\, R_V 
\end{align*}
is called 
the  {\em associated variety} 
of $V$ (\cite{Ara12}).
The associated variety ${X}_V$ is a fundamental invariant of $V$ that captures some important properties
of the vertex algebra $V$ itself (see e.g.\ \cite{BeiFeiMaz,Zhu96,AbeBuhDon04,Miy04,Ara12,Ara09b,A2012Dec,AM15,AM16}.

As a Poisson variety, the associated variety of a vertex algebra %$V$
is a finite disjoint union of smooth analytic Poisson manifolds, 
and it is stratified by its symplectic leaves. 
%Therefore there is a well-defined symplectic leaf though any point of it.
The case where the associated variety has finitely many symplectic leaves 
is particularly interesting.
This happens for instance when  
$V$ is an admissible affine vertex algebra\footnote{that is,
the simple 
affine vertex algebras associated with admissible representations of affine Kac-Moody algebras.} (\cite{Ara09b}),
when $V$ is the simple affine vertex algebra associated with a simple 
Lie algebra that 
belongs to the
the Deligne exceptional series 
 \cite{De96} at level $k=-h^{\vee}/6-1$ (\cite{AM15}) 
with $h^\vee$ the dual Coxeter number, 
or when $V$ is the (generalized) Drinfeld-Sokolov reduction 
(\cite{FF90,KacRoaWak03}) of the latter affine vertex algebras
 provided that  it is nonzero  (\cite{Ara09b}).
This is also expected to happen for the vertex algebras 
obtained from four dimensional $N=2$ superconformal field theories (\cite{BeeLemLie15}),
where the associated variety is expected to coincide with the spectrum 
of the chiral ring of the 
Higgs branch of the four dimensional theory.
Of course, it also happens when the associated variety of $V$ is a point,
that is, when $V$ is {\em lisse} (or $C_2$-cofinite).
 
In our previous paper \cite{AM16}
we conjectured that, under reasonable assumptions on the 
vertex algebra $V$, the associated variety $X_V$ is irreducible
if it has only finitely many symplectic leaves (see Conjecture \ref{Conj:equidim}). 

One of the aims of this paper is 
to verify  this irreducibility conjecture for the known cases 
where $X_V$ has finitely many symplectic leaves.
It turns out this is a subtle problem for 
the associated varieties of (conjecturally simple) $W$-algebras,
and some deep results on the geometry of nilpotent orbits 
are needed 
for the verifications.
More precisely, the associated variety in question has the form
\begin{align*}
\Slo_{\O,f}:=\overline{\O} \cap \Slo_f,
\end{align*}
where $\O$ is a nilpotent orbit in a simple Lie algebra $\fing$ 
and $\Slo_f $ is {\em the} Slodowy slice at a nilpotent
element $f$ of $\fing$. 
The variety $\Slo_{\O,f}$ is called a {\em nilpotent Slodowy slice} (\cite{FJLS15}). 
It is not irreducible in general.
On the other hand, not every nilpotent Slodowy slice appears 
as the associated variety of some $W$-algebra.
We investigate in detail the irreducibility of the variety $\Slo_{\O,f}$ 
that appears as the associated variety of a $W$-algebra 
and confirm the irreducibility  for all the known cases.

Another aim of this paper is to provide 
new examples of vertex algebras
whose associated variety has finitely many symplectic leaves.
We do this by proving a conjecture stated in our previous article 
\cite[Conjecture 2]{AM16} 
(Theorem~\ref{Th:AMconj}) and also by showing that 
the associated variety of 
the simple affine vertex algebra associated with $\mf{so}_{2r+1}$
at level $-2$ is the short nilpotent orbit closure (Theorem~\ref{Th:short}). 
Theorems~\ref{Th:AMconj} and \ref{Th:short} 
give new examples of non-admissible affine vertex algebras 
whose associated variety is contained in the nilpotent cone.  

\subsection*{Acknowledgments} 
We thank very much Daniel Juteau for
his explanations about branchings and 
nilpotent Slodowy slices. 
We thank Drazen Adamovi\'{c} for inviting us
to the University 
of Zagreb in June 2016, and the CIRM Luminy
for its hospitality during our stay 
as ^^ ^^ Research in pairs" in August, 2016.

The first named author is supported 
 by JSPS KAKENHI Grant Numbers 25287004 and 26610006.
The second named author 
is supported by the ANR Project GeoLie Grant number ANR-15-CE40-0012. 

\section{Known examples of vertex algebras whose associated variety 
has finitely many symplectic leaves}
Let $V$ be a vertex algebra,
and let 
$$V\ra (\End V)[[z,z^{-1}]],\quad a\mapsto a(z)=\sum_{n\in \Z}a_{(n)}z^{-n-1}, 
$$be the state-field correspondence.
The Zhu's $C_2$-algebra is by definition the quotient space 
$R_V=V/C_2(V)$,
where $C_2(V)=\on{span}_{\C}\{a_{(-2)}b\mid a,b\in V\}$,
equipped with the Poisson algebra structure given by
$$\bar a. \bar b=\overline{a_{(-1)}b},\quad \{\bar a,\bar b\}=\overline{a_{(0)}b},$$
for $a,b \in V$ with $\bar a := a+C_2(V)$. 
The associated variety $X_V$ of $V$ is the reduced scheme of the spectrum of $R_V$, 
$X_V=\on{Specm}(R_V)$.

In \cite[Conjecutre 1 (2)]{AM16}, we stated the following conjecture.
   \begin{Conj}\label{Conj:equidim}
    Let $V{=\oplus_{d \geq 0} V_d}$ be a simple, finitely strongly generated 
    (i.e., $R_V$ is finitely generated), positively graded 
    conformal vertex operator algebra such that $V_0=\C$.
Assume that $X_V$ has finitely many symplectic leaves. 
    Then $X_V$  is irreducible.
   \end{Conj}

In the following we list 
the known examples of vertex algebras whose associated 
variety has finitely many symplectic leaves.

\subsection{Lisse vertex algebras}
Recall \cite{Zhu96} that a vertex algebra is {\em lisse} (or {\em $C_2$-cofinite}) if 
$\dim X_V=0$. 
Under the assumptions on $V$ of Conjecture~\ref{Conj:equidim}, 
$R_V=\oplus_{d \geq 0} (R_V)_d$ and $(R_V)_0=\C$, so 
$X_V$ is conic, that is, 
admits a $\C^*$-action that is contracting to a point. 
Therefore, under the assumptions on $V$ of Conjecture~\ref{Conj:equidim}, 
$V$ is lisse if and only if $X_V=\{\rm{point}\}$.
Hence, if so, $X_V$ is a trivial symplectic variety and is irreducible.

\subsection{Affine vertex algebras}
Let $\fing$ be a complex simple Lie algebra, and 
$\affg$ the affine Kac-Moody algebra associated with $\fing$:
\begin{align*}
\affg=\fing\*\C[t,t^{-1}]\+ \C K \+ \C D,
\end{align*}
where the commutation relations are given by
\begin{align*}
[x\*t^m,y\*t^n]=[x,y]\*t^{m+n}+m(x|y)\delta_{m+n,0}K,\quad 
[K,\affg]=0,\quad
[D,x\*t^m]=mx\*t^{m}
\end{align*}
for $x,y\in \fing$, $m,n\in \Z$. 
Here $(~|~)=\displaystyle{\frac{1}{2h^\vee}\times}$Killing form of $\fing$.
For $k\in \C$, set 
\begin{align*}
V^k(\fing)=U(\affg)\*_{U(\fing\* \C[t]\+ \C K\+ \C D)}\C_k,
\end{align*}
where $\C_k$ is the one-dimensional representation of $\fing\* \C[t]\+ \C K\+ \C D$
on which $K$ acts as multiplication by $k$ and $\fing\* \C[t]\+ \C D$ acts trivially.
As it is well-known
 $V_k(\fing)$
 is naturally a vertex algebra, 
 called the {\em universal affine vertex algebra} associated with $\fing$ at level $k$.
 The unique simple (graded) quotient  $V_k(\fing)$ is called
the  {\em simple affine vertex algebra} associated with $\fing$ at level $k$. 
 
 We have $X_{V^k(\fing)}=\fing^*$,
 and therefore
  the associated variety $X_{V_k(\fing)}$ is a  Poisson subscheme of $\fing^*$
which is $G$-invariant and conic, where $G$ is the adjoint group of $\g$.
Thus, identifying $\g$ with $\g^*$ through $(~|~)$, the symplectic leaves 
of $X_{V_k(\fing)}$ are exactly the adjoint orbits 
of $G$ in $X_{V_k(\fing)}$, 
and $X_{V_k(\fing)}$  has finitely many symplectic leaves if and only if $X_{V_k(\fing)}$ is contained in the nilpotent cone
$\mathcal{N}$ of $\fing$.
%Known cases when this occurs are the following.
%(We do not 
%claim that these are exhaustive. )

We list in Table \ref{Tab:known-cases} the known pairs $(\g,k)$  
where $X_{V_k(\fing)}\subset \mc{N}$.
 We do not 
claim that the list is exhaustive. 

 {\footnotesize\begin{table}[ht]
 \begin{center}
\begin{tabular}{|c|c|c|c|}
\hline
&type of $\g$ &$k$ & $X_{V_k(\fing)}$ \\[0.25em]
\hline &&& \\[-0.75em]
(1) & any & $-h^\vee$ &  $\mc{N}$ \\[0.25em]
\hline &&& \\[-0.75em]
(2) &any & admissible & $\overline{\O}_q$\\[0.25em]
\hline &&& \\[-0.75em]
(3) & $G_2$ & $-1$ & $\overline{\mathbb{O}_{min}}$ \\[0.25em]
\hline &&& \\[-0.75em]
(4) & $D_4$, $E_6$, $E_7$, $E_8$ & 
$k \in \Z$, 
$-\frac{h^{\vee}}{6}-1\leq k\leq -1$ & $\overline{\mathbb{O}_{min}}$\\[0.25em]
\hline &&& \\[-0.75em]
(5) & $D_r$ with $r\geq 5$ & $-2,-1$ & $\overline{\mathbb{O}_{min}}$ \\[0.25em]
\hline &&& \\[-0.75em]
(6) & $D_r$ with $r$ an even integer & $2-r$ & 
$\overline{ \mathbb{O}_{(2^{r-2},1^4)}}$ \\[0.25em]
\hline &&& \\[-0.75em]
(7) & $B_r$ & $-2$ &$\overline{ \mathbb{O}_{short}}$ \\[0.25em]
\hline
\end{tabular}
 \end{center}
\vspace{.125cm}
\caption{Known pairs ($\fing,k$) for which 
$X_{V_k(\fing)} \subset\mc{N}$} \label{Tab:known-cases}
 \end{table}}
 In case~(1) of Table \ref{Tab:known-cases}, $V_{-h^{\vee}}(\fing)$ does not satisfy the assumption
of Conjecture \ref{Conj:equidim} since it
 is not conformal,
 but the irreducibility of $X_{V_{-h^{\vee}}(\fing)}$ holds.
 
In case~(2) of Table \ref{Tab:known-cases}, $\O_q$ is a nilpotent orbit of $\fing$ described by 
Tables 2--10 of \cite{Ara09b} which only depends on the denominator 
$q$ of the admissible level $k \in \Q$. 

In cases~(3)--(5) of Table \ref{Tab:known-cases}, 
$\O_{min}$ is the minimal nilpotent orbit of $\fing$.

The statement of the case  (6) of Table \ref{Tab:known-cases} 
was conjectured in \cite{AM16} 
and will be proven in Section \ref{sec:proof-of-AM-conjecuture}. 
Here $\mathbb{O}_{(2^{r-2},1^4)}$ denotes the nilpotent orbit 
of $\mf{so}_{2r}$ corresponding to the partition $(2^{r-2},1^4)$ 
of $2r$ (see Section \ref{sec:normal}).  

In case~(7) of Table \ref{Tab:known-cases}, 
$\O_{short}$ is the 
unique {\em short} nilpotent orbit in $\mf{so}_{2r+1}$.
Here by {\em short nilpotent orbit} we mean the orbit of a short nilpotent element,
and a nilpotent element $e$ of $\g$ is called {\em short} if
for $(e,f,h)$ an $\mf{sl}_2$-triple, 
$$\fing=\fing_{-1}\+\fing_0\+\fing_1,$$
where $\fing_j=\{x\in \fing\mid [h,x]=2jx\}$.
The fact that $X_{V_{-2}(\mf{so}_{2r+1})}=\overline{\O_{short}}$ is new and 
will be proven in Theorem \ref{Th:short}. 

 In all the cases of Table~\ref{Tab:known-cases}, 
 $X_{V_k(\fing)}$ is the closure of some nilpotent orbit in $\mc{N}$ 
 and so Conjecture~\ref{Conj:equidim} holds.

 \subsection{Simple $W$-algebras}
Let $\W^k(\fing,f)$ be the {\em affine $W$-algebra} associated with 
a nilpotent element $f$ of $\fing$ 
defined by the generalized quantized Drinfeld-Sokolov reduction
\cite{FF90,KacRoaWak03}:
$$\W^k(\fing,f)=H^{0}_{DS,f}(V^k(\fing)).$$
Here $H^{\bullet}_{DS,f}(M)$ denotes the BRST 
cohomology of the  generalized quantized Drinfeld-Sokolov reduction
associated with $f \in \mc{N}$ with coefficients in 
a $V^k(\fing)$-module $M$.

We have \cite{De-Kac06,Ara09b} a natural isomorphism
$R_{\W^k(\fing,f)}\cong \C[\Slo_{f}]$ of Poisson algebras, so that
\begin{align*}
 X_{\W^k(\fing,f)}= \Slo_{f}.
\end{align*}
Here $ \Slo_{f}$ is the Slodowy slice at $f$ as in the introduction 
(cf.\ Section \ref{subsec:nilslo}). 
It has a natural Poisson structure induced from that of $\fing^*$ (\cite{GanGin02}).  

Let  $\W_k(\fing,f)$ be  the unique simple (graded) quotient of
$\W^k(\fing,f)$.
Then
$X_{\W_k(\fing,f)}$
 is a $\C^*$-invariant Poisson
subvariety of the Slodowy slice $\Slo_f$.

Let $\mc{O}_k$ be the category $\mc{O}$ of
$\affg$ at level $k$.
We have a functor
\begin{align*}
 \mc{O}_k\ra \W^k(\fing,f)\on{-Mod}
 ,\quad M\mapsto
 H^0_{DS,f}(M),
\end{align*}
where
$\W^k(\fing,f)\on{-Mod}$ denotes the category
of $\W^k(\fing,f)$-modules.

Let $\on{KL}_k$ be the full subcategory of $\mc{O}_k$ consisting of
objects $M$ on which $\fing$ acts  locally finitely.
Note that $V^k(\fing)$ and $V_k(\fing)$ are objects of $\on{KL}_k$.
 \begin{Th}[{\cite{Ara09b}}]\label{Th:W-algebra-variety}
 \begin{enumerate}
  \item $H_{DS,f}^i(M)=0$ for all $i\ne 0$, $M\in
	\on{KL}_k$.
	In particular, the functor
	$\on{KL}_k\ra \W^k(\fing,f)\on{-Mod}$, $M\mapsto
	H_{DS,f}^{0}(M)$, is exact.

\item
     For any quotient $V$ of $V^k(\fing)$, 
    \begin{align*}
X_{H^{0}_{DS,f}(V)}=X_{V}\cap  \Slo_{f}.
%\label{:eq:var-of-reduction}
\end{align*}
In particular  
\begin{enumerate}
\item $H_{DS,f}^{0}(V)
	  \ne 0$ if and only if
	  $\overline{G.f}\subset X_V$, 
\item $H_{DS,f}^{0}(V)$ is lisse if $X_V=\overline{G.f}$.
\end{enumerate}
  \label{item:intersection}
    \end{enumerate}
\end{Th}

By Theorem \ref{Th:W-algebra-variety} (1),
$H^{0}_{DS,f}(V_k(\fing))$ is a quotient vertex algebra of $\W^k(\fing,f)$ if it is nonzero.
Conjecturally \cite{KacRoaWak03,KacWak08},
we have
\begin{align*}
\W_k(\fing,f)\cong H^{0}_{DS,f}(V_k(\fing))\text{ provided that }H^{0}_{DS,f}(V_k(\fing))\ne 0.
\end{align*}
(This conjecture has been verified in many cases \cite{Ara05,Ara07,Ara08-a}.)

By Theorem \ref{Th:W-algebra-variety} (2),
 when $X_{V_k(\fing)} \subset \mc{N}$, 
then $X_{H^{0}_{DS,f}(V_k(\fing))}$ is contained in $\Slo_f \cap \mc{N}$
and so has finitely many symplectic leaves. 

One of the purposes of the paper is to investigate 
 the irreducibility of $X_V$ 
for $V=H^{0}_{DS,f}(V_k(\fing))$ 
with
$X_{V_k(\fing)}\subset \mc{N}$.

\section{Branching and nilpotent Slodowy slices} 
We collect in this section the results that we need about branchings 
and nilpotent Slodowy slices. 
We refer to \cite[Chap.~III, \S4.3]{EGAIII} for the original 
definition of {\em unibranchness}, and to 
\cite{KraftProcesi82} or \cite{FJLS15} 
for further details on branchings 
and nilpotent Slodowy slices. 
 
\subsection{Smoothly equivalent singularities, cross sections} 
Consider two varieties $X,Y$ and two points $x \in X$, $y \in Y$. 
The singularity of $X$ at $x$ is called {\em smoothly equivalent} 
to the singularity of $Y$ at $y$ if there is a variety $Z$, a point $z \in Z$ 
and two maps $\varphi \colon Z \to X$, $\psi\colon Z \to Y$, 
%$$\xymatrix{Z \ar[r]^{\varphi} \ar[d]^{\psi} & X \\ Y& }$$
such that $\varphi(z)=x$, $\psi(z)=y$, and $\varphi$ and $\psi$ 
are smooth in $z$ (\cite{Hesselink76}). This clearly defines an equivalence relation 
between pointed varieties $(X,x)$. We denote the equivalence 
class of $(X,x)$ by $\Sing(X,x)$. 

Various geometric properties of $X$ at $x$ only depends on the 
equivalence class $\Sing(X,x)$, for example: {\em smoothness, 
normality, seminormality (cf.~\cite[\S16.1]{KraftProcesi82}), unibranchness (cf.~\S\ref{subsec:branching}), Cohen-Macaulay, rational singularities.}

Assume that the algebraic group $G$ acts regularly on the variety $X$. 
Then $\Sing(X,x)=\Sing(X,x')$ if $x$ and $x'$ belongs to the same 
$G$-orbit $\O$. In this case, we denote the equivalence class also 
by $\Sing(X,\O)$. 

A {\em cross section} (or {\em 
transverse slice}) at the point $x \in X$ is defined to be a locally 
closed subvariety $S \subset X$ such that $x \in S$ and the map 
$$G \times S \longrightarrow X, \quad (g,s) \longmapsto g.s,$$
is smooth at the point $(1,x)$. 
We have $\Sing(S,x)=\Sing(X,x)$. 
%There is a natural way to construct cross section for affine $G$-varieties 
%$X$ \cite[\S12.4]{KraftProcesi82}. 
%Choose a $G$-equivariant closed embedding 
%$X \hookrightarrow V$ in some vector space $V$ with linear $G$-action 
%and a complement $N$ to the tangent space $T_x(Gx)$ in $V$. 
%Define $S:=(N+x)\cap X$ (scheme theoretic intersection). 
%Then  $G \times S \to X, \, (g,s) \mapsto gs$ is smooth at the point 
%$(1,x)$ since $G\times (N+x) \to V, \, (g,n+x) \mapsto g(n+x)$ is smooth 
%at $(1,x)$ %(\cite[Chap.~IV, 17.11.1]{EGA}) 
%and 
%$$\xymatrix{G \times (N+x) \ar[r]  & V \\
%G\times S \ar@{^{(}->}[u] \ar[r] & X \ar@{^{(}->}[u]}
%$$
%is a fibre product. Hence $S$ is reduced in $X$ and so $S$, as 
%a variety, is a cross section at $x$. The construction implies that 
%$x$ is an isolated point in $S \cap Gx$. 
%Assuming that $X$ is irreducible (or equidimensional) 
%we get $\dim_x S= \on{codim}(Gx,X)$. 

In the case where $X$ is the closure of some nilpotent $G$-orbit 
of $\fing$, there is a natural choice of a cross section 
as we explain next subsection. 

\subsection{Nilpotent Slodowy slice}\label{subsec:nilslo}
Let $\O,\O'$ be two nonzero nilpotent orbits of $\fing$ and pick $f \in \O'$. 
By the Jacobson-Morosov Theorem, we can embed $f$ into an 
$\mf{sl}_2$-triple $(e,h,f)$ of $\fing$. 
The affine space 
\begin{align*}
\Slo_f:=f+\g^{e}
\end{align*}
is a transverse slice of $\fing$ at $e$, 
called the {\em Slodowy slice} associated with $(e,h,f)$. 
There is a $\C^*$-action on $\Slo_f$ contracting to $f$ (cf.~\cite{GanGin02}). 
%Let $\tilde{\rho} \colon \C^* \to G$ 
%be the one-parameter subgroup of $G$ defined by 
%$$ \forall \, t\in \C^*, \; 
%\forall\,x \in \g,\qquad\tilde{\rho}(t)(x)=t^{2}\rho(t)x,$$
%where $\rho \colon \C^* \to G$ is the one-parameter subgroup of $G$ 
%generated by $\ad h$. 
%Then $\tilde{\rho}$ stabilizes the Slodowy slice $\Slo_f$,  
%and the $\C^*$-action of $\tilde{\rho}$ on $\Slo_f$ 
%is contracting to $f$. 
The variety 
\begin{align*}
\Slo_{\O,f}:=\overline{\O}\cap \Slo_f 
\end{align*}
is then a transverse slice of $\overline{\O}$ at $f$, 
which we call, following the terminology of \cite{FJLS15}, 
a {\em nilpotent Slodowy slice}. 

Note that $\Slo_{\O,f}=\{f\}$ if and only if $\O=G.f$. 
Moreover, since the $\C^*$-action %of $\tilde{\rho}$ 
on $\Slo_f$ 
is contracting to $f$ and stabilizes $\Slo_{f,\O}$, 
$\Slo_{\O,f}=\varnothing$ if and only if $G.f \not \subseteq \overline{\O}$. 
Hence we can assume that $\O' \subseteq \overline{\O}$, 
that is, $\O' \leq \O$ for the Chevalley order on nilpotent orbits. 
The variety $\Slo_{\O,f}$ is equidimensional, and 
\begin{align*}
\dim \Slo_{\O,f}= \on{codim}(\O',\overline{\O}).
\end{align*}

Since any two $\mf{sl}_2$-triples containing  $f$ 
are conjugate by an element of the isotropy group 
of $f$ in $G$, 
the isomorphism type of $\Slo_{\O,f}$ is independent of 
the choice of such $\mf{sl}_2$-triples. 
Moreover, the isomorphism type of $\Slo_{\O,f}$ is 
independent of the choice of $f \in \O'$. 
By focussing on $\Slo_{\O,f}$, we reduce the study 
of $\Sing(\overline{\O},\O')$ to the study 
of the singularity of $\Slo_{\O,f}$ at $f$. 

The variety $\Slo_{\O,f}$ is not always irreducible.  
We are now interested in sufficient conditions for that $\Slo_{\O,f}$ 
is irreducible. 

\subsection{Branching} \label{subsec:branching} 
Let $X$ be an irreducible algebraic variety, and $x \in X$. 
We say that $X$ is {\em unibranch at $x$} if the normalization 
$\pi \colon (\tilde{X},x) \to (X,x)$ of $(X,x)$ 
is locally a homeomorphism at $x$ 
\cite[\S2.4]{FJLS15}. Otherwise, we say that $X$ {\em has branches 
at $x$} 
and the number of branches of $X$ at $x$ 
is the number of connected components 
of $\pi^{-1}(x)$ \cite[\S5,(E)]{BeynonSpaltenstein84}. 

As it is explained in \cite[Section 2.4]{FJLS15}, 
the number of irreducible components of 
$\Slo_{\O,f}$ is equal to the number of branches of $\overline{\O}$ 
at $f$. 
%Indeed, since the $\C^*$-action of $\tilde{\rho}$ on $\Slo_{\O,f}$ is 
%contracting to $f$, $\Slo_{\O,f}$ is connected and each 
%of its irreducible components is 
%unibranch at $f$. 

If an irreducible algebraic variety $X$ is normal, 
then it is obviously unibranch at any point $x \in X$. 
Hence we obtain the following result. 

\begin{Lem} \label{Lem:normal} 
Let $\O,\O'$ be nilpotent orbits of $\fing$, with $\O' \leq \O$ 
and $f \in \O'$. 
If $\overline{\O}$ is normal, then $\Slo_{\O,f}$ is 
irreducible. 
\end{Lem}

The converse is not true. For instance, there is no branching 
in type $G_2$ but one knowns that the nilpotent orbit $\tilde{A}_1$ 
of $G_2$ 
of dimension $8$ is not normal \cite{LevasseurSmith88}. 

The number of branches of $\overline{\O}$ 
at $f$, and so the number of irreducible components of 
$\Slo_{\O,f}$, can be determined from the tables of Green functions 
in \cite{Shoji80,BeynonSpaltenstein84}, as discussed in 
\cite[Section 5,(E)-(F)]{BeynonSpaltenstein84}.
%These tables are given in the recent preprint \cite{FJLS} (private communictation).
%We can also check the computations 
%using the program GAP3 \cite{GAP}. 
We indicate in Table \ref{Tab:branching} the nilpotent orbits $\O$ 
which have branchings in types $F_4$, $E_6$, 
$E_7$ and $E_8$ (there is no branching in type $G_2$). 
The nilpotent orbits are labelled using the Bala-Carter classification.  
{\footnotesize\begin{table}[h]
 \begin{center}
\begin{tabular}{|l|l|} 
\hline
Type $F_4$ &$C_3$, $C_3(a_1)$. \\
\hline
Type $E_6$ &$A_4$, $2 A_2$, $A_2+A_1$. \\
\hline
Type $E_7$ &$D_6(a_1)$, $(A_5)''$, $A_4$, $A_3+A_2$, 
$D_4(a_1)+A_1$. \\
\hline
Type $E_8$ &$E_7(a_1)$, $E_6$, $E_6(a_1)$, $E_7(a_4)$, 
$A_6$, $D_6(a_1)$, $D_5+A_1$, $E_7(a_5)$, 
$A_4$, $A_3+A_2$, $D_4$, \\ 
& $D_4(a_1)$, $A_3+A_1$, 
$2 A_2+A_1$. \\
\hline
\end{tabular}
\end{center}
\vspace{.125cm}
\caption{Branching in exceptional cases} \label{Tab:branching}
\end{table}}

We indicate below the (conjectural) list a non-normal nilpotent 
orbit closures in the exceptional types. These results 
are extracted from \cite{LevasseurSmith88,Kraft89,Broer98a,Broer98b,Sommers03}. 
The list is known to be exhaustive for the types $G_2$, $F_4$ and $E_6$. 
It is only conjecturally exhaustive for the types $E_7$ 
and $E_8$. 

{\footnotesize\begin{table}[h]
 \begin{center}
\begin{tabular}{|l|l|} 
\hline
Type $G_2$ &$\tilde{A}_1$. \\
\hline
Type $F_4$ &
$C_3$, 
$C_3(a_1)$, 
$\tilde{A}_2+A_1$, 
$\tilde{A}_2$, 
$B_2$. \\
\hline
Type $E_6$ &
$A_4$, 
$A_3+A_1$, 
$A_3$, 
$2 A_2$, 
$A_2+A_1$. \\
\hline
Type $E_7$ &
$D_6(a_1)$, 
$D_6(a_2)$, 
$(A_5)''$, 
$A_4$, 
$A_3+A_2$, 
$D_4(a_1)+A_1$, 
$A_3+2A_1$, \\
& $(A_3+A_1)'$,  
$(A_3+A_1)''$, 
$A_3$. \\
\hline
Type $E_8$ & 
$E_7(a_1)$, 
$E_7(a_2)$, 
$D_7(a_1)$, 
$E_7(a_3)$,  
$E_6$, 
$D_6$,
$E_6(a_1)$, 
$E_7(a_4)$, 
$D_6(a_1)$, \\
& $A_6$, 
$D_5+A_1$, 
$E_7(a_5)$, 
$E_6(a_3)+A_1$, 
$D_6(a_2)$, 
$D_5(a_1)+A_2$, 
$A_5+A_1$, \\
& $D_5$, 
$E_6(a_3)$, 
$D_4+A_2$, 
$D_5(a_1)+A_1$,  
$A_5$, 
$D_5(a_1)$, 
$D_4+A_1$, 
$A_4$, \\
& $A_3+A_2$, 
$A_3+2 A_1$, 
$D_4$,   
$D_4(a_1)$
$A_3+A_1$,  
$2A_1+A_1$,  
$A_3$.\\
\hline
\end{tabular}
\end{center}
\vspace{.125cm}
\caption{Non-normal nilpotent 
orbits in exceptional cases} \label{Tab:non-normal}
\end{table}}

\section{Normality of nilpotent orbit closures in Lie algebras 
of classical type} \label{sec:normal} 

In view of Lemma \ref{Lem:normal}, 
we recall in this section some useful results about the 
normality of nilpotent orbit closures in the classical types. 
The normality question in this setting is now completely answered
(\cite{KraftProcesi79,KraftProcesi79,Sommers05}).
 
First of all, by \cite{KraftProcesi79}, if $\fing=\sl_n$, 
then all nilpotent orbit closures are normal. 
So we only focus on the orthogonal and symplectic Lie algebras. 
We assume in the rest of the section that $\g$ is either the Lie algebra 
$\mf{o}_n$ of the orthogonal group $O(n)$, or 
the Lie algebra $\so_n$ of the special orthogonal group $SO(n)$, 
or the Lie algebra $\sp_n$ of the symplectic group $SP(n)$. 

\subsection{Notations} 
We fix in this subsection some notations relative to nilpotent 
orbits in simple Lie algebras of classical type. 
Our main references are 
\cite{KraftProcesi82,CMa}. 
We follow the notations of \cite[Appendix]{MorYu16}; 
see therein for more details.  

Let $n\in \N^{*}$, and denote by $\P(n)$ the set 
of partitions of $n$. 
As a rule, unless otherwise specified, we write 
an element $\bs{\lambda}$ 
of $\P(n)$ 
as a decreasing sequence $\bs{\lambda}=(\lambda_1,\ldots,\lambda_r)$ 
omitting the zeroes. 
Thus, 
$$
\lambda_1 \geq \cdots \geq \lambda_r \geq 1\quad 
\text{ and }\quad  
\lambda_1 + \cdots + \lambda_r  = n.
$$ 
%We shall denote the dual partition of a partition 
%$\bs{\lambda}\in\P(n)$ by $^t\!{\bs{\lambda}}$.

Let us denote by $\geqslant$ the partial order on $\P(n)$ relative 
to dominance. More precisely, given
$\bs{\lambda} = ( \lambda_{1},\cdots ,\lambda_{r}),
\bs{\eta} = (\mu_{1},\dots ,\mu_{s}) \in \P (n)$,  
we have $\bs{\lambda}\ \geqslant \bs{\eta}$ if 
$
\sum_{i=1}^{k} \lambda_{i} \geqslant \sum_{i=1}^{k} \mu_{i}
$
for $1\leqslant k\leqslant \min (r,s)$.

\subsubsection{Case $\sl_n$}
By \cite[Theorem~5.1.1]{CMa}, nilpotent orbits of $\sl_{n}$ are 
parametrized by $\P(n)$. For $\bs{\lambda}\in\P(n)$, 
we shall denote by 
$\O_{\bs{\lambda}}$ the corresponding nilpotent orbit of 
$\sl_n$. 
%, and if we write
%$^t\!{\bs{\lambda}} = ( d_{1}, \dots ,d_{s})$, then
%$$
%\dim \O_{\bs{\lambda}} = n^{2} - \sum_{i=1}^{s} d_{i}^{2}.$$ 
If $\bs{\lambda}, \bs{\eta} \in \P (n)$, then 
$\O_{\bs{\eta}} \subset \overline{\O_{\bs{\lambda}}}$ if and only if
$\bs{\eta} \leqslant \bs{\lambda}$.

\subsubsection{Cases $\mf{o}_n$ and $\so_n$}
For $n\in\N^*$, set 
$$
\P_{1}(n):=\{\bs{\lambda} \in \P(n)\; ; \; \text{number 
of parts of each even 
number is even}\}.
$$ 
By \cite[Theorems 5.1.2 and 5.1.4]{CMa}, 
nilpotent orbits of $\so_{n}$ 
are parametrized by $\P_1(n)$, with the exception that each 
{\em very even} 
partition $\bs{\lambda} \in\P_{1}(n)$ (i.e., $\bs{\lambda}$ has only even parts) 
corresponds to two nilpotent orbits.
For $\bs{\lambda}\in \P_1(n)$, not very even, we shall denote by 
$\O_{\bs{1,\lambda}}$, 
or simply by $\O_{\bs{\lambda}}$ when there is no possible confusion, 
the corresponding nilpotent orbit of $\so_n$. 
For very even $\bs{\lambda}\in \P_1(n)$, we shall denote by 
$\O_{1,\bs{\lambda}}^{I}$ 
and $\O_{1,\bs{\lambda}}^{I\!I}$ the two corresponding nilpotent orbits 
of $\so_n$.  
In fact, their union forms a single $O(n)$-orbit. 
Thus nilpotent orbits of $\mf{o}_{n}$ 
are parametrized by $\P_1(n)$. 

%Let $\bs{\lambda} =(\lambda_{1},\dots ,\lambda_{r})\in\P_1(n)$ 
%and $^t\!{\bs{\lambda}} = (d_{1},\dots ,d_{s})$, then 
%$$
%\dim \O_{1,\bs{\lambda}}^{\bullet} = \frac{n(n-1)}{2}
%-\frac{1}{2}\left( \sum_{i=1}^{s} d_{i}^{2}
%-  \#\{ i ; \lambda_{i} \hbox{ odd} \}  \right), 
%$$
%where $\O_{1,\bs{\lambda}}^{\bullet}$ is either 
%$\O_{1,\bs{\lambda}}$, $\O_{1,\bs{\lambda}}^{I}$ or 
%$\O_{1,\bs{\lambda}}^{II}$ according 
%to whether $\bs{\lambda}$ is very even or not. 
%Using the same notations, 
If $\bs{\lambda},\bs{\eta}\in\P_{1} (n)$, then 
$\overline{\O_{1,\bs{\eta}}^{\bullet}} \subsetneq
\overline{\O_{1,\bs{\lambda}}^{\bullet}}$ if and only if
$\bs{\eta} < \bs{\lambda}$, 
where $\O_{1,\bs{\lambda}}^{\bullet}$ is either 
$\O_{1,\bs{\lambda}}$, $\O_{1,\bs{\lambda}}^{I}$ or 
$\O_{1,\bs{\lambda}}^{II}$ according 
to whether $\bs{\lambda}$ is very even or not.

Given $\bs{\lambda} \in\P(n)$, there exists a unique 
$\bs{\lambda}^{+}
\in \P_{1}(n)$ such that $\bs{\lambda}^{+} \leq \bs{\lambda} $, and if 
$\bs{\eta}\in\P_{1}(n)$ verifies $\bs{\eta} \leq\bs{\lambda} $, then 
$\bs{\eta}\leq \bs{\lambda} ^{+}$.
%More precisely, let $\bs{\lambda}=(\lambda_{1},\dots ,\lambda_{n})$ 
%(adding zeroes if necessary). 
%If $\bs{\lambda} \in \P_{1}(n)$, then $\bs{\lambda} ^{+} =\bs{\lambda} $. 
%Otherwise if $\bs{\lambda} \not\in \P_{1}(n)$, set
%$$
%\bs{\lambda} ' = (\lambda_{1},\dots ,\lambda_{r},\lambda_{r+1}-1,
%\lambda_{r+2},\dots ,
%\lambda_{s-1},\lambda_{s}+1,\lambda_{s+1},\dots, \lambda_{n}),
%$$
%where $r$ is maximum such that $(\lambda_{1},\dots,\lambda_{r})\in 
%\P_{1}(\lambda_{1}+\cdots + \lambda_{r})$, and $s$ is the index of 
%the first even part
%in $(\lambda_{r+2},\dots ,\lambda_{n})$. Note that $r=0$ if such a 
%maximum does not 
%exist, while $s$ is always defined. If $\bs{\lambda} '$ is not in 
%$\P_{1}(n)$, then we repeat 
%the process until we obtain an element of $\P_{1}(n)$ which 
%will be our $\bs{\lambda}^{+}$.

\subsubsection{Case $\sp_{n}$} 
For $n\in\N^*$, set 
$$
\P_{-1}(n):=\{\bs{\lambda} \in \P(n)\; ; \; \text{number of parts of each odd  
number is even}\}.
$$ 
By \cite[Theorem~5.1.3]{CMa}, nilpotent orbits of $\sp_{n}$ 
are parametrized by $\P_{-1}(n)$. 
For $\bs{\lambda}=(\lambda_{1},\dots ,\lambda_{r})\in \P_{-1}(n)$, 
we shall denote by 
$\O_{-1,\bs{\lambda}}$, 
or simply by $\O_{\bs{\lambda}}$ when there is no possible confusion,  
the corresponding nilpotent orbit of $\sp_{n}$. 
%and if we write $^t\!{\bs{\lambda}} = (d_{1},\dots ,d_{s})$, then 
%$$
%\dim \O_{-1,\bs{\lambda}} = \frac{n(n+1)}{2} - \frac{1}{2}\left( \sum_{i=1}^{s} d_{i}^{2}
%+  \#\{ i ; \lambda_{i} \hbox{ odd} \}  \right).
%$$
As in the case of $\sl_{n}$, if $\bs{\lambda}, \bs{\eta} \in \P_{-1} (n)$, then 
$\O_{-1,\bs{\eta}} \subset \overline{\O_{-1,\bs{\lambda}}}$ if and only if
$\bs{\eta} \leqslant \bs{\lambda}$. 

Given $\bs{\lambda}\in\P(n)$, there exists a unique 
$\bs{\lambda}^{-} \in \P_{-1}(n)$ 
such that $\bs{\lambda} ^{-} \leq \bs{\lambda} $, 
and if $\bs{\eta}\in\P_{-1}(n)$ verifies 
$\bs{\eta} \leq \bs{\lambda} $, then 
$\bs{\eta}\leq \bs{\lambda} ^{-}$. 
%The construction of 
%$\bs{\lambda} ^{-}$ is the same as in the orthogonal case 
%except that $s$ is the index 
%of the first odd part in $(\lambda_{r+2},\dots,\lambda_{n})$.

\begin{Def} \label{Def:degeneration}
Let $\bs{\lam} \in \P_{\eps}(n)$. 
An {\em $\eps$-degeneration} of $\bs{\lam}$ is 
an element $\bs{\eta} \in \P_{\eps}(n)$ such that 
$\O_{\eps,\bs{\eta}}\subseteq \overline{\O_{\eps,\bs{\lam}}}$, 
that is, $ \bs{\eta}\leq \bs{\lam}$.
\end{Def}

\subsection{Some general facts}
Let $\O$ be a nilpotent orbit of $\fing$. 
Recall that the singular locus of $\overline{\O}$ is 
$\overline{\O} \setminus \O$. 
This was shown by Namikawa \cite{Namikawa04} using results 
of Kaledin and Panyushev \cite{Kaledin06,Panyushev}; 
see \cite[Section 2]{Henderson} for a recent review. 
This result also follows from 
Kraft and Procesi's work in the 
classical types \cite{KraftProcesi81,KraftProcesi82}, and 
from the main theorem of \cite{FJLS15} 
in the exceptional types.

\begin{Th}[{\cite[Theorem 1]{KraftProcesi82}}] \label{Th:normal-codim} 
Let $\O$ be a nilpotent orbit in $\mf{o}_n$ or $\sp_n$. 
\begin{enumerate}
\item 
$\overline{\O}$ is normal if and only if it is unibranch. 
\item 
$\overline{\O}$ is normal if and only if it is normal 
in codimension 2. 

\end{enumerate}
\end{Th}
In particular, $\overline{\O}$ is normal if it does not contain a 
nilpotent orbit $\O' \leq \O$ of codimension 2. 
Theorem \ref{Th:normal-codim} does not hold if $\g=\so_{2n}$ 
and if $\O=\O_{1,\bs{\lam}}$, with $\bs{\lam}$ very even. 

\begin{Th}[\cite{Hesselink79}] \label{Th:normal-partition} 
Let $\eps\in\{-1,1\}$ and $\bs{\lam} \in\P_\eps(n)$.  
\begin{enumerate}
\item Assume $\g=\mf{o}_n$. 
If $\lam_1+\lam_2\leq 4$, then $\overline{\O_{1,\bs{\lam}}}$ is normal.
\item Assume $\g=\sp_n$. 
If $\lam_1\leq 2$, then $\overline{\O_{-1,\bs{\lam}}}$ is normal.  
\end{enumerate}
\end{Th}
Theorem~\ref{Th:normal-partition} does not hold if $\g=\so_{2n}$ 
and if $\O=\O_{1,\bs{\lam}}$ with $\bs{\lam}$ very even. 

\subsection{Minimal $\eps$-degeneration}
We present in this subsection the combinatorial 
method developed in \cite{KraftProcesi82} 
to determine the equivalence 
class $\Sing(\overline{\O_{\eps,\bs{\lam}}},\O_{\eps,\bs{\eta}})$, 
for $\eps\in\{-1,1\}$ and $\bs{\eta}<\bs{\lam}$.  

\begin{Def}[{\cite[Definition 3.1]{KraftProcesi82}}]\label{Def:deg-min} 
An $\eps$-degeneration $\bs{\eta}\leq\bs{\lam}$ is called 
{\em minimal} if $\bs{\eta}\not= \bs{\lam}$ and there is 
no $\bs{\nu} \in \P_\eps(n)$ such 
$\bs{\eta} <\bs{\nu}< \bs{\lam}$ (i.e., $\bs{\nu} < \bs{\lam}$ 
are adjacent in the ordering on $\P_\eps(n)$). 
\end{Def}
In geometrical terms this means that the orbit $\O_{\eps,\bs{\eta}}$ 
is open in the complement of $\O_{\eps,\bs{\lam}}$ 
in $\overline{\O_{\eps,\bs{\lam}}}$. 

By Theorem \ref{Th:normal-codim}, 
for the normality question, it is enough to consider minimal 
$\eps$-degeneration of codimension 2 (except for the 
very even nilpotent orbits in $\so_{2n}$). 
Kraft and Procesi introduced a combinatorial equivalence on 
$\eps$-degenerations \cite{KraftProcesi82}.

\begin{Pro}[{\cite[Proposition 3.2]{KraftProcesi82}}] \label{Pro:deg}
Let $\bs{\eta}\leq\bs{\lam}$ be an $\eps$-degeneration. 
Assume that for two integers $r$ and $s$ the first $r$ rows 
and the first $s$ columns of $\bs{\eta}$ and $\bs{\lam}$ 
coincide and that $(\lam_1,\ldots,\lam_r)$ is an element of 
$\P_{\eps}(\lam_1+\cdots+\lam_r)$. 
Denote by $\bs{\eta}'$ and $\bs{\lam}'$ 
be the partitions obtained by erasing these common rows and columns 
of $\bs{\eta}$ and $\bs{\lam}$ respectively and put 
$\eps':=(-1)^s\eps$. 
Then $\bs{\eta}'\leq\bs{\lam}'$ is an $\eps'$-degeneration and
$$\on{codim}(\O_{\eps',\bs{\eta}'},\overline{\O_{\eps',\bs{\lam}'}}) 
=\on{codim}(\O_{\eps,\bs{\eta}},\overline{\O_{\eps,\bs{\lam}}}).$$
\end{Pro} 

\begin{Def}[{\cite[Definition 3.3]{KraftProcesi82}}]\label{Def:deg-irred}  
In the setting of Proposition \ref{Pro:deg} we say that the
$\eps$-degeneration $\bs{\eta}\leq\bs{\lam}$ {\em is obtained from the 
$\eps'$-degeneration $\bs{\eta}'\leq\bs{\lam}'$ 
by adding rows and columns}.

An $\eps$-degeneration $\bs{\eta}\leq\bs{\lam}$ 
is called {\em irreducible} if it cannot be obtained by adding 
rows and columns in a non trivial way. 
\end{Def}

In the setting of Proposition \ref{Pro:deg}, 
when we obtain an irreducible pair $(\bs{\eta}',\bs{\lam}')$, 
such a pair is called the {\em type of 
$(\O_{\eps,\bs{\eta}},\O_{\eps,\bs{\lam}})$}. 

\begin{Rem}
\begin{enumerate}
\item In the previous setting, 
$\bs{\eta}'\leq\bs{\lam}'$ is minimal if and only if 
$\bs{\eta}\leq\bs{\lam}$ is minimal. 
\item Any $\eps$-degeneration is obtained in a unique way 
from an irreducible $\eps'$-degeneration by adding rows and columns.
\end{enumerate}
\end{Rem}

So for the classification of the minimal $\eps$-degenerations, 
one needs to describe the minimal irreducible $\eps$-degenerations. 
They are given in \cite[Table 3.4]{KraftProcesi82}. 
We reproduce it (see Table \ref{Tab:deg}) since it will be our main tool in Section \ref{sec:conclusions}. 
In the sixth line, ^^ ^^ $\on{codim}$" refers to the codimension 
of $\O_{\eps,\bs{\eta}}$ in $\overline{\O_{\eps,\bs{\lam}}}$. 
The meaning of the last column of the table is 
explained in \cite[\S\S14.2 and 14.3]{KraftProcesi82}. 
For our purpose, what is important is that, 
except for the type $e$, 
%for the type $e$, 
%the singularity of $\overline{\O_{\eps,\bs{\lam}}}$ 
%in $\O_{\eps,\bs{\eta}}$ is smoothly equivalent to 
%an isolated surface singularity of type $A_k\cup A_k$,  
%and this singularity is the non-normal union of two 
%surface singularities of type $A_k$ 
%meeting transversally in a singular point 
%\cite[\S0.9, Theorem 9]{KraftProcesi82}. 
%For all the other types, 
the singularity of $\overline{\O_{\eps,\bs{\lam}}}$ 
in $\O_{\eps,\bs{\eta}}$ is normal. 

{\footnotesize\begin{table}[h]
 \begin{center}
\begin{tabular}{|c|c|c|c|c|c|c|} 
\hline &&&&&& \\[-0.75em]
Type & Lie algebra & $\eps$ & $\bs{\lam}$ & $\bs{\eta}$ & 
$\on{codim}$ %_{\O_{\eps',\bs{\lam}'}}(\O_{\eps',\bs{\eta}'})$ 
& 
$\Sing(\overline{\O_{\eps,\bs{\lam}}},\O_{\eps,\bs{\eta}})$ \\[0.25em]
\hline
$a$ & $\sp_2$ & -1 & $(2)$ & $(1,1)$ & 2 & $A_1$  \\
$b$ & $\sp_{2n}$, $n>1$ & -1 & $(2n)$ & $(2n-1,2)$ 
& $2$ & $D_{n+1}$  \\
$c$ & $\so_{2n+1}$, $n>0$ & 1 & $(2n+1)$ & $(2n-1,1,1)$ 
& $2$ & $A_{2n-1}$ \\
$d$ & $\sp_{4n+2}$, $n>0$ & -1 & $(2n+1,2n+1)$ 
& $(2n,2n,2)$ 
& $2$ & $A_{2n-1}$  \\
$e$ & $\so_{4n}$, $n>0$ & 1 & $(2n,2n)$ 
& $(2n-1,2n-1,1,1)$ 
& $2$ & $A_{2n-1}\cup A_{2n-1}$ \\
$f$ & $\so_{2n+1}$, $n>1$ & 1 & $(2,2,1^{2n-1})$ 
& $(1^{2n+1})$ 
& $4n-4$ & $b_{n}$  \\
$g$ & $\sp_{2n}$, $n>1$ & -1 & $(2,1^{2n-2})$ 
& $(1^{2n})$ 
& $2n$ & $c_{n}$ \\
$h$ & $\so_{2n}$, $n>2$ & 1 & $(2,2,1^{2n-4})$ 
& $(1^{2n})$ 
& $4n-6$ & $d_{n}$ \\
\hline
\end{tabular}
\end{center}
\vspace{.125cm}
\caption{Irreducible minimal $\eps$-degenerations} 
\label{Tab:deg}
\end{table}}

The main result is the following. 

\begin{Th}[{\cite[Theorem 12.3]{KraftProcesi82}}] \label{Th:deg}
Let $\bs{\eta}\leq\bs{\lam}$ be 
the $\eps$-degeneration obtained from the 
$\eps'$-degeneration $\bs{\eta}'\leq\bs{\lam}'$ 
by adding rows and columns. 
Then $$\Sing(\overline{\O_{\eps,\bs{\lam}}},\O_{\eps,\bs{\eta}})
=\Sing(\overline{\O_{\eps',\bs{\lam}'}},\O_{\eps',\bs{\eta}'}).$$
\end{Th}

Since very even nilpotent orbits in $\so_{2n}$ do not appear as 
known associated varieties of vertex algebras, 
we do not give here the results about the normality 
question in this case: 
the question was partially answered in 
\cite[Theorem 17.3]{KraftProcesi82}, the remaining 
cases were dealt with in \cite{Sommers05}. 

\section{Irreducibility of nilpotent Slodowy slices 
as associated varieties} \label{sec:conclusions}
Using the results of the previous sections we are now in position 
to check the following: {\em for any nilpotent 
orbit $\O$ of $\fing$ appearing in Table \ref{Tab:known-cases}, 
and for any $f \in \overline{\O}$, $\Slo_{\O,f}$ is irreducible.}

We prove the statement by cases in Table \ref{Tab:known-cases}. 

$\ast$ Case (1). 
We have $\O=\mc{N}$. Here the result is well-known 
(or results from Lemma \ref{Lem:normal}). 

$\ast$ Cases (2). 
First we assume that $\fing$ is of exceptional type. 
Then we have to check that none of the nilpotent 
orbits appearing in Tables 4--10 of \cite{Ara09b} appears in Table \ref{Tab:branching}. 
We readily verify that it is true and so we are done. 
Actually, none of these nilpotent 
orbits (except the nilpotent orbit $\tilde{A}_1$ in 
$G_2$) even appears in Table~\ref{Tab:non-normal}. 

For the classical types, we have more work to do. 
The result will follow from~\S\ref{subsec:admissible}. 

$\ast$ Cases (3), (4) or (5). 
Here $\O=\O_{min}$ and so the result is clear since either $f \in \O_{min}$ 
or $f=0$. 

$\ast$ Cases (6). 
Here $\O=\mathbb{O}_{(2^{r-2},1^4)}$. 
According to Theorem \ref{Th:normal-partition} (1), 
$\overline{\O}$ is normal and so it is unibranch by Theorem \ref{Th:normal-codim}~(1). 
Hence for any $f \in \overline{ \mathbb{O}_{(2^{r-2},1^4)}}$, 
$\Slo_{\O,f}$ is irreducible.

\subsection{Admissible cases in the classical types}\label{subsec:admissible}
In this subsection, we give the necessary data 
to verify that all nilpotent orbits of Tables 2--3 
of \cite{Ara09b}\footnote{Not every nilpotent orbits 
of these tables appear as associated variety 
of admissible vertex algebras. Namely, in Table 2, the case 
where $q$ is even in $\so_{2r+1}$ does not appear in such a way 
so we do not consider that case.\label{footnote:adm} }
have normal closures. The case of $\sl_{n}$ 
is clear, so we only consider the cases where $\g=\sp_{2r}$,  
$\g=\so_{2r+1}$ and $\g=\so_{2r}$. 

First of all, we observe that no very even nilpotent orbits 
in $\mf{so}_{2r}$ appear in these tables. 
Then results of Section \ref{sec:normal} apply 
and our strategy is the following. 

Let $\bs{\lam}\in \P_{\eps}(n)$ 
be {\em anyone} (as in the footnote \ref{footnote:adm}) 
of the partitions appearing in Tables 2--3 of \cite{Ara09b}. 
Then consider all minimal $\eps$-degenerations $\bs{\eta}$ 
of $\bs{\lam}$ such that 
$$\on{codim}(\O_{\eps,\bs{\eta}},\overline{\O_{\eps,\bs{\lam}}})=2.$$  
It may happen that there are several such a $\bs{\eta}$ for a given 
$\bs{\lam}$, or that there is no such a $\bs{\eta}$ (in this case, 
$\overline{\O_{\eps,\bs{\lam}}}$ does not contain any nilpotent 
orbit in codimension 2, and so it is normal).  

Then the type of the singularity of 
$\overline{\O_{\eps,\bs{\lam}}}$ in  $\O_{\eps,\bs{\eta}}$ 
is obtained following the receipt of Proposition~\ref{Pro:deg}: 
we erase common rows and common columns in  
 $\bs{\eta}$ and $\bs{\lam}$ in order to get 
an irreducible pair
$(\bs{\lam}',\bs{\eta}')$ and we set $\eps'=(-1)^s$, 
where $s$ is the number of common columns. 
The type of 
$(\O_{\eps',\bs{\lam}'},\O_{\eps',\bs{\eta}'})$ is described in Table \ref{Tab:deg} 
and we conclude thanks to Theorem \ref{Th:deg}. 

According to Theorem \ref{Th:normal-partition}, 
we can assume that $\lam_1 \geq 3$ ($\lam_1$ will be $q$  
is our tables). 

For each $\bs{\lam}$ as in Tables 1,2,3 
of \cite{Ara09b} (as in the footnote \ref{footnote:adm}), 
we present in the Tables \ref{Tab:data-sp}, \ref{Tab:data-so-1} 
and \ref{Tab:data-so-2} 
all possible $\bs{\eta}$ as above, the number $\eps'$, 
the corresponding irreducible pair  
$(\bs{\lam}',\bs{\eta}')$, 
and the type of $(\O_{\eps',\bs{\lam}'},\O_{\eps',\bs{\eta}'})$ 
following Table~\ref{Tab:deg}. 

\subsubsection{Case $\g=\sp_{2r}$} 
Here $\eps=-1$. 
According to \cite[Tables 2-3]{Ara09b}, 
the different possibilities for $\bs{\lam}\in\P_{-1}(2r)$ are the following: 
\begin{description}
\item[I] 
$\bs{\lam}=({q,\ldots,q},s)$, $0 \leq  
s \leq q-1$, $q$ odd, $s$ even. 
\item[II] 
$\bs{\lam}=({q,\ldots,q},q-1,s)$, 
$0 \leq  
s \leq q-1$, $q$ odd, $s$ even. 
\item[III] 
$\bs{\lam}=(q,\ldots,q,s)$, 
$0 \leq s\leq q-1$, $q$ even, $s$ even. 
\item[IV]  
$\bs{\lam}=(q+1,{q,\ldots,q},s)$, 
$0 \leq s\leq q-1$, $q$ odd, $s$ even. 
\item[V]  
$\bs{\lam}=(q+1,{q,\ldots,q},
q-1,s)$, 
$2\leq s\leq q-1$, $q$ odd, $s$ even. 
\end{description}

{\footnotesize\begin{table}[h]
 \begin{center}
\begin{tabular}{|c|c|c|c|c|c|c|} 
\hline
$\bs{\lam}$ & $q,s$ & $\bs{\eta}$ & 
$\eps'$ & $\bs{\lam}'$ & $\bs{\eta}'$ & Type \\
\hline
\bf{I}
& $0 \leq s\leq q-3$
& $(q,\ldots,q,q-1,q-1,s+2)$
& $-1$ 
& $(q-s,q-s)$ 
& $(q-s-1,q-s-1,2)$ 
& $d$ \\
& $4 \leq s \leq q-1$ 
& $(q,\ldots,q,s-2,2)$  
& $-1$ 
& $(s)$ 
& $(s-2,2)$ 
& $b$ \\
& $s=2$ 
&  $(q,\ldots,q,1,1)$ 
& $-1$ 
& $(2)$ 
& $(1,1)$ 
& $a$ \\
\bf{II} 
& $0 \leq s\leq q-5$ 
&  $(q,\ldots,q,q-3,s+2)$ 
& $-1$ 
& $(q-1-s)$ 
& $(q-3-s,2)$ 
& $b$ \\
& $s=q-3$ 
& $(q,\ldots,q,q-2,s+1)$  
& $-1$ 
& $(2)$ 
& $(1,1)$ 
& $a$ \\
& $4 \leq s\leq q-1$
& $(q,\ldots,q,q-1,s-2,2)$ 
& $-1$ 
& $(s)$ 
& $(s-2,2)$ 
& $b$ \\
& $s=2$ 
&  $(q,\ldots,q,q-1,1,1)$ 
& $-1$ 
& $(2)$ 
& $(1,1)$ 
& $a$ \\ 
\bf{III} 
& $0 \leq s \leq q-2$ 
&  $(q,\ldots,q,q-2,s+2)$ 
& $-1$ 
& $(q-s)$ 
& $(q-s-2,2)$ 
& $b$ \\
& $s=q-2$ 
&  $(q,\ldots,q,q-1,s+1)$ 
& $-1$ 
& $(2)$ 
& $(1,1)$ 
& $a$ \\
& $4 \leq s\leq q-1$
& $(q,\ldots,q,s-2,2)$ 
& $-1$ 
& $(s)$ 
& $(s-2,2)$ 
& $b$ \\
& $s=2$ 
&  $(q,\ldots,q,1,1)$ 
& $-1$ 
& $(2)$ 
& $(1,1)$ 
& $a$ \\
\bf{IV}
& $0 \leq s\leq q-3$
& $(q+1,q,\ldots,q,q-1,q-1,s+2)$
& $-1$ 
& $(q-s,q-s)$ 
& $(q-s-1,q-s-1,2)$ 
& $d$ \\
& $4 \leq s \leq q-1$ 
& $(q+1,q,\ldots,q,s-2,2)$  
& $-1$ 
& $(s)$ 
& $(s-2,2)$ 
& $b$ \\
& $s=2$ 
&  $(q+1,q,\ldots,q,1,1)$ 
& $-1$ 
& $(2)$ 
& $(1,1)$ 
& $a$ \\
\bf{V} 
& $0 \leq s\leq q-5$ 
&  $(q+1,q,\ldots,q,q-3,s+2)$ 
& $-1$ 
& $(q-1-s)$ 
& $(q-3-s,2)$ 
& $b$ \\
& $s=q-3$ 
& $(q+1,q,\ldots,q,q-2,s+1)$  
& $-1$ 
& $(2)$ 
& $(1,1)$ 
& $a$ \\
& $4 \leq s\leq q-1$
& $(q+1,q,\ldots,q,q-1,s-2,2)$ 
& $-1$ 
& $(s)$ 
& $(s-2,2)$ 
& $b$ \\
& $s=2$ 
&  $(q+1,q,\ldots,q,q-1,1,1)$ 
& $-1$ 
& $(2)$ 
& $(1,1)$ 
& $a$ \\ 
\hline
\end{tabular}
\end{center}
\vspace{.125cm}
\caption{Data for $\g=\sp_{2r}$} 
\label{Tab:data-sp}
\end{table}}

\subsubsection{Case $\g=\so_{2r+1}$}
Here $\eps=1$. 
According to \cite[Tables 2-3]{Ara09b} 
and the footnote~\ref{footnote:adm}, 
the different possibilities for 
$\bs{\lam}\in\P_1(2r+1)$ are the following: 
\begin{description}
\item[I] 
$\bs{\lam}=({q,\ldots,q},s)$, 
$0 \leq s\leq q$, $q$ odd with even multiplicity, $s$ odd.  
\item[II] 
$\bs{\lam}=({q,\ldots,q},s,1)$, 
$0 \leq s\leq q-1$, $q$ odd with odd multiplicity, $s$ odd.  
%\item[III] 
%$\bs{\lam}=(q+1,\underbrace{q,\ldots,q}_{\rm{even}})$, 
%$q$ even.  \textcolor{red}{This does not appear}
%\item[IV] 
%$\bs{\lam}=(q+1,\underbrace{q,\ldots,q}_{\rm{even}},s,1)$, 
%$1 \leq s\leq q-1$, $q$ even, $s$ odd.   \textcolor{red}{This does not appear}
%\item[V] 
%$\bs{\lam}=(q+1,\underbrace{q,\ldots,q}_{\rm{even}},q-1,s)$, 
%$1 \leq s\leq q-1$, $q$ even, $s$ odd.  \textcolor{red}{This does not appear}
\item[III] 
$\bs{\lam}=({q,\ldots,q},s)$, 
$0 \leq s\leq q-1$, $q$ even, $s$ odd.  
\item[IV] 
$\bs{\lam}=({q,\ldots,q},q-1,s,1)$, 
$0 \leq s\leq q-1$, $q$ even, $s$ odd.  
\end{description}

{\footnotesize\begin{table}[h]
 \begin{center}
\begin{tabular}{|c|c|c|c|c|c|c|} 
\hline
$\bs{\lam}$ & $q,s$ & $\bs{\eta}$ & 
$\eps'$ & $\bs{\lam}'$ & $\bs{\eta}'$ & Type \\
\hline
\bf{I}
& $0 \leq s\leq q-4$
& $(q,\ldots,q,q-2,s+2)$
& $-1$ 
& $(q-s)$ 
& $(q-s-2,2)$ 
& $b$ \\
& $3 \leq s \leq q$ 
& $(q,\ldots,q,s-2,1,1)$  
& $1$ 
& $(s)$ 
& $(s-2,1,1)$ 
& $c$ \\
& $s=q-2$ 
&  $(q,\ldots,q,q-1,q-1)$ 
& $-1$ 
& $(2)$ 
& $(1,1)$ 
& $a$ \\
\bf{II} 
& $0 \leq s\leq q-4$
& $(q,\ldots,q,q-2,s+2,1)$
& $-1$ 
& $(q-s)$ 
& $(q-s-2,2)$ 
& $b$ \\
& $5 \leq s \leq q-1$ 
& $(q,\ldots,q,s-2,3)$  
& $-1$ 
& $(s-1)$ 
& $(s-3,2)$ 
& $b$ \\
& $s=3$ 
&  $(q,\ldots,q,2,2)$ 
& $-1$ 
& $(2)$ 
& $(1,1)$ 
& $a$ \\
%\bf{III} 
%&  
%&  $(q+1,q,\ldots,q,q-1,q-1,1,1)$ 
%& $1$ 
%& $(q,q)$ 
%& $(q-1,q-1,1,1)$ 
%& $e$ \\
%\bf{IV}
%& $0 \leq s\leq q-3$
%& $(q+1,q,\ldots,q,q-1,q-1,s+2,1)$
%& $-1$ 
%& $(q-s,q-s)$ 
%& $(q-s-1,q-s-1,2)$ 
%& $d$ \\
%& $4 \leq s \leq q-1$ 
%& $(q+1,q,\ldots,q,s-2,2)$  
%& $-1$ 
%& $(s)$ 
%& $(s-2,2)$ 
%& $b$ \\
%& $s=2$ 
%&  $(q+1,q,\ldots,q,1,1)$ 
%& $-1$ 
%& $(2)$ 
%& $(1,1)$ 
%& $a$ \\
%\bf{V} 
%& $1 \leq s\leq q-5$ 
%&  $(q+1,q,\ldots,q,q-3,s+2)$ 
%& $-1$ 
%& $(q-1-s)$ 
%& $(q-3-s,2)$ 
%& $b$ \\
%& $s=q-3$ 
%& $(q+1,q,\ldots,q,q-2,s+1)$  
%& $-1$ 
%& $(2)$ 
%& $(1,1)$ 
%& $a$ \\
%& $3 \leq s\leq q-1$
%& $(q+1,q,\ldots,q,q-1,s-2,1,1)$ 
%& $1$ 
%& $(s)$ 
%& $(s-2,1,1)$ 
%& $c$ \\
\bf{III}
& $0 \leq s\leq q-3$
& $(q,\ldots,q,q-1,q-1,s+2)$
& $-1$ 
& $(q-s,q-s)$ 
& $(q-1-s,q-s-1,2)$ 
& $d$ \\
& $3 \leq s \leq q$ 
& $(q,\ldots,q,s-2,1,1)$  
& $1$ 
& $(s)$ 
& $(s-2,1,1)$  
& $c$ \\ 
\bf{IV} 
& $1 \leq s\leq q-5$ 
&  $(q,\ldots,q,q-3,s+2,1)$ 
& $-1$ 
& $(q-1-s)$ 
& $(q-3-s,2)$ 
& $b$ \\
& $s=q-3$ 
& $(q,\ldots,q,q-2,s+1,1)$  
& $-1$ 
& $(2)$ 
& $(1,1)$ 
& $a$ \\
& $5 \leq s\leq q-1$
& $(q,\ldots,q,q-1,s-2,3)$ 
& $-1$ 
& $(s-1)$ 
& $(s-3,2)$ 
& $b$ \\
& $s=3$
& $(q,\ldots,q,q-1,2,2)$ 
& $-1$ 
& $(2)$ 
& $(1,1)$ 
& $a$ \\
\hline
\end{tabular}
\end{center}
\vspace{.125cm}
\caption{Data for $\g=\so_{2r+1}$} 
\label{Tab:data-so-1}
\end{table}}

\subsubsection{Case $\g=\so_{2r}$}
Here $\eps=1$. 
According to \cite[Table 2]{Ara09b}, 
the different possibilities for $\bs{\lam}\in\P_1(2r)$ are the following: 
\begin{description}
\item[I] 
$\bs{\lam}=({q,\ldots,q},s)$, 
$0\leq s\leq q$, $q$ odd with odd multiplicity, $s$ odd.  
\item[II] 
$\bs{\lam}=({q,\ldots,q},s,1)$, 
$0\leq s\leq q-1$, $q$ odd with even multiplicity, $s$ odd.  
\item[III] 
$\bs{\lam}=(q+1,{q,\ldots,q},s)$, 
$0\leq s\leq q-1$, $q$ even, $s$ odd. 
\item[IV] 
$\bs{\lam}=(q+1,{q,\ldots,q},q-1,s,1)$, 
$0\leq s\leq q-1$, $q$ even, $s$ odd. 
\end{description}

{\footnotesize\begin{table}[h]
 \begin{center}
\begin{tabular}{|c|c|c|c|c|c|c|} 
\hline
$\bs{\lam}$ & $q,s$ & $\bs{\eta}$ & 
$\eps'$ & $\bs{\lam}'$ & $\bs{\eta}'$ & Type \\
\hline
\bf{I}
& $0 \leq s\leq q-4$
& $(q,\ldots,q,q-2,s+2)$
& $-1$ 
& $(q-s)$ 
& $(q-s-2,2)$ 
& $b$ \\
& $3 \leq s \leq q$ 
& $(q,\ldots,q,s-2,1,1)$  
& $1$ 
& $(s)$ 
& $(s-2,1,1)$ 
& $c$ \\
& $s=q-2$ 
&  $(q,\ldots,q,q-1,s+1)$ 
& $-1$ 
& $(2)$ 
& $(1,1)$ 
& $a$ \\
\bf{II} 
& $0 \leq s\leq q-4$
& $(q,\ldots,q,q-2,s+2,1)$
& $-1$ 
& $(q-s)$ 
& $(q-s-2,2)$ 
& $b$ \\
& $5 \leq s \leq q-1$ 
& $(q,\ldots,q,s-2,3)$  
& $-1$ 
& $(s-1)$ 
& $(s-3,2)$ 
& $b$ \\
& $s=3$ 
&  $(q,\ldots,q,2,2)$ 
& $-1$ 
& $(2)$ 
& $(1,1)$ 
& $a$ \\
\bf{III} 
& $3 \leq s \leq q-1$  
&  $(q+1,q,\ldots,q,s-2,1,1)$ 
& $1$ 
& $(s)$ 
& $(s-2,1,1)$ 
& $c$ \\
& $0 \leq s \leq q-3$ 
&  $(q+1,q,\ldots,q,q-1,q-1,s+2)$ 
& $-1$ 
& $(q-s,q-s)$ 
& $(q-1-s,q-1-s,2)$ 
& $d$ \\
& $s=2$ 
&  $(q+1,q,\ldots,q,1,1)$ 
& $-1$ 
& $(2)$ 
& $(1,1)$ 
& $a$ \\
\bf{IV}
& $0 \leq s\leq q-5$
& $(q+1,q,\ldots,q,q-3,s+2,1)$
& $-1$ 
& $(q-1-s)$ 
& $(q-3-s,2)$ 
& $b$ \\
& $s=q-3$ 
& $(q+1,q,\ldots,q,q-2,s+1,1)$  
& $-1$ 
& $(2)$ 
& $(1,1)$ 
& $a$ \\
& $5 \leq s\leq q-1$
& $(q+1,q,\ldots,q,q-1,s-2,3)$
& $-1$ 
& $(s-1)$ 
& $(s-3,2)$ 
& $b$ \\
& $s=3$
& $(q+1,q,\ldots,q,q-1,2,2)$
& $-1$ 
& $(2)$ 
& $(1,1)$ 
& $a$ \\
\hline
\end{tabular}
\end{center}
\vspace{.125cm}
\caption{Data for $\g=\so_{2r}$} 
\label{Tab:data-so-2}
\end{table}}

\begin{Rem}
The above verifications show that, except for the non-normal 
nilpotent orbit $\tilde{A}_1$ of dimension 8 in $G_2$ that appears as 
the associated variety of simple affine vertex algebras of
type $G_2$ at admissible levels with denominator $2$ \cite{Ara09b}, 
all the other nilpotent orbits appearing in Table \ref{Tab:known-cases} 
are actually normal. 
\end{Rem}

\section{Proof of Conjecture 2 of \cite{AM16}}\label{sec:proof-of-AM-conjecuture}
In this section we prove the following assertion.
\begin{Th}[{Conjecture 2 of \cite{AM16}}]\label{Th:AMconj}
Let $\fing=\mf{so}_{2r}$ with $r$ even, $r \geq 4$, 
and let $k=2-r$.
Then 
$$X_{V_k(\fing)}=\overline{\mathbb{O}_{(2^{r-2},1^4)}}.$$
\end{Th}
\begin{proof}
Let $\fing=\mf{so}_{2r}$ with $r$ even, $r\geq 4$.
Let $\Delta=\{\pm \varepsilon_i\pm  \varepsilon_j   
\; |\; 1\leq i,j\leq {r }, i\not= j\}$ be the root system of $\g$ 
and take $\Delta_+=\{\varepsilon_i\pm  \varepsilon_j 
\; |\; 1\leq i < j\leq {r }\}$ for the set of positive roots. 

By Theorem 1.3 of \cite{AM16}
we know that 
$X_{V_k(\fing)}\subset \overline{\mathbb{O}_{(2^{r-2},1^4)}}$.
Hence it is sufficient to show that 
$X_{V_k(\fing)}\supset \overline{\mathbb{O}_{(2^{r-2},1^4)}}$.
By Theorem \ref{Th:W-algebra-variety} (1),
this is equivalent to that
$$H_{DS,f}^0(V_k(\fing))\ne 0\quad \text{ for } 
\quad f\in \mathbb{O}_{(2^{r-2},1^4)}.$$

Let $f\in  \mathbb{O}_{(2^{r-2},1^4)}$,
and let $(e,f,h)$ be  an $\mf{sl}_2$-triple in $\fing$.
The weighted Dynkin diagram of $f$ is the following.
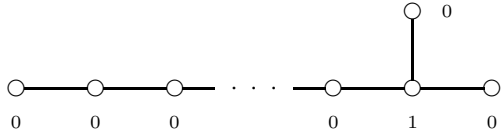
\vspace{-0.5cm}
$$\begin{Dynkin}
\Dspace\Dspace\Dspace\Dspace\Dspace\Dbloc{\Dcirc\Dsouth\Dtext{r}{0}}
\Dskip
\Dbloc{\Dcirc\Deast\Dtext{b}{0}}
\Dbloc{\Dcirc\Dwest\Deast\Dtext{b}{0}}
\Dbloc{\Dcirc\Dwest\Deast\Dtext{b}{0}}
\Dbloc{\Ddots}
\Dbloc{\Dcirc\Dwest\Deast\Dtext{b}{0}}
\Dbloc{\Dcirc\Dwest\Deast\Dnorth\Dtext{b}{1}}
\Dbloc{\Dcirc\Dwest\Dtext{b}{0}}
%\Dskip
%\Dspace\Dspace\Dspace\Dspace\Dspace\Dspace
%\Dbloc{\Dcirc\Dnorthwest\Dtext{t}{0}}
\end{Dynkin}$$
    Thus,  we may assume that $h=\varpi_{2r-2}^\vee$,
    where $\varpi_i^\vee$ is the $i$-th fundamental coweight of $\fing$.
    We have
    \begin{align}
\fing=\fing_{-1}\+\fing_{-1/2}\+ \fing_{0}\+\fing_{1/2}\+\fing_{1},
\label{eq:almost-short}
\end{align}
where $\fing_j=\{x\in \fing\mid [h,x]=2jx\}$.
Set $\Delta_j=\{\alpha\in \Delta\mid x_{\alpha}\in \fing_j\}$,
so that $\Delta=\bigsqcup_j \Delta_j$.
    
Set
\begin{align*}
 D_h=D+\frac{1}{2}h,
\end{align*}
and put
$V_k(\fing)_{[d]}=\{v\in V_k(\fing)\mid D_h v=d v\}$.
Since 
$\alpha(D_h)\geq 0$ for any positive real root $\alpha$ of $\affg$ 
by \eqref{eq:almost-short},
we have
\begin{align*}
V_k(\fing)=\bigoplus_{d\in \frac{1}{2}\Z_{\geq 0}}V_k(\fing)_{[d]}.
\end{align*}

The operator  $D_h$ extends to the
grading operator of $\W^k(\fing,f)$ 
as it commutes with the differential of the complex associated 
with the Drinfeld-Sokolov reduction
(\cite{KacRoaWak03,Ara05}).
Thus,
\begin{align*}
H_{DS,f}^0(V_k(\fing))=\bigoplus_{d\in \frac{1}{2}\Z_{\geq 0}}H_{DS,f}^0(V_k(\fing))_{[d]},
\end{align*}
where 
$H_{DS,f}^0(V_k(\fing))_{[d]}=\{c\in H_{DS,f}^0(V_k(\fing))\mid D_h c=d c\}$.

Now, let
$\mf{a}$ be the sualgebra of $\affg$
generated by $x_{\alpha}\otimes t$,  $\alpha\in \Delta_{-1}$,
$x_{\beta}$,  $\beta\in \Delta_{0}$,
$x_{\gamma} \otimes t^{-1}$,  $\gamma\in \Delta_{1}$.
Then $\mf{a}$ is isomorphic to 
\begin{align*}
\fing_{even}:=\fing_{-1}\+\fing_0\+\fing_{1}\subset \fing,
\end{align*}
and $\mf{a}$ acts on the each homogeneous subspace 
$V_k(\fing)_{[d]}$ as its elements commute with $D_h$.
We have  $$\fing_{even}\cong \mf{so}_{2(r-2)}\+ \mf{sl}_2\+\mf{sl}_2,$$
where $\mf{so}_{2(r-2)}$ is the subalgebra of $\fing$ corresponding to the roots
$\pm \epsilon_i\pm \epsilon_j$ with $1\leq i< j\leq r-2$,
and $ \mf{sl}_2\+\mf{sl}_2$ is the subalgebra corresponding to the roots 
$\pm \epsilon_{r-1}\pm \epsilon_r$.

Let $\mf{p}=\mf{l}\+\mf{m}$ be the parabolic subalgebra of $\mf{a}$ 
whose nilradical $\mf{m}$ is
$\{x_{\alpha}\otimes t\mid \alpha\in \Delta_{-1}\}$.
The Levi subalgebra $\mf{l}$ is isomorphic to
 $$\fing_0\cong \mf{gl}_{r-2}\+ \mf{sl}_2\+\mf{sl}_2.$$

The $\mf{a}$-module $V_k(\fing)_{[0]}$
is a 
 highest weight representation of 
$\mf{a}$ with highest weight $((2-r)\varpi_{2(r-2)},0,0)$
under the isomorphism $\mf{a}\cong \mf{so}_{2(r-2)}\+ \mf{sl}_2\+\mf{sl}_2$.
This weight verifies the conditions of the Jantzen's simplicity criterion
\cite{Jan77} (see also \cite[Section 9.13]{Humphreys}). Thus from 
Jantzen's simplicity criterion 
one finds that
\begin{align}\label{eq:Jantzen}
V_k(\fing)_{[0]}\cong U(\mf{a})\*_{U(\mf{p})}\C |0\ket.
\end{align}

The nilpotent element $f\otimes 
t\in \mf{m}\subset \mf{a}$ belongs to $ \mf{so}_{2(r-2)}$
under the isomorphism $\mf{a}\cong \mf{so}_{2(r-2)}\+ \mf{sl}_2\+\mf{sl}_2$,
and is a very even nilpotent element that corresponds to the partition $(2^{r-2})$.
The weighted Dynkin diagram of $f \otimes t$ as nilpotent 
element of $\so_{2(r-2)}$ is given by
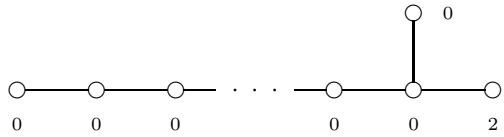
\vspace{-0.5cm}
$$\begin{Dynkin}
\Dspace\Dspace\Dspace\Dspace\Dspace\Dbloc{\Dcirc\Dsouth\Dtext{r}{0}}
\Dskip
\Dbloc{\Dcirc\Deast\Dtext{b}{0}}
\Dbloc{\Dcirc\Dwest\Deast\Dtext{b}{0}}
\Dbloc{\Dcirc\Dwest\Deast\Dtext{b}{0}}
\Dbloc{\Ddots}
\Dbloc{\Dcirc\Dwest\Deast\Dtext{b}{0}}
\Dbloc{\Dcirc\Dwest\Deast\Dnorth\Dtext{b}{0}}
\Dbloc{\Dcirc\Dwest\Dtext{b}{2}}
%\Dskip
%\Dspace\Dspace\Dspace\Dspace\Dspace\Dspace
%\Dbloc{\Dcirc\Dnorthwest\Dtext{t}{0}}
\end{Dynkin}$$

From the definition of the Drinfeld-Sokolov reduction we find that
$$H_{DS,f}^0(V_k(\fing))_{[0]}\cong H_0(\mf{m}_-, V_k(\fing)_{[0]}\* \C_{\chi}),$$
where 
$\mf{m}_-$ is the opposite algebra of $\mf{m}$ 
spanned by $x_{\alpha}\otimes t^{-1}$,  $\alpha\in \Delta_1$,
so that 
$\mf{a}=\mf{m}_-\+\mf{p}$,
and $\C_{\chi}$ is the one-dimensional representation of 
$\mf{m}_-$ defined  by the character
$x\mapsto (x|f\otimes t)$.
Since $V_k(\fing)_{[0]}$ is a free $U(\mf{m}_-)$-module by 
\eqref{eq:Jantzen},
it follows that
$ H_0(\mf{m}_-, V_k(\fing)_{[0]}\* \C_{\chi})\cong  H_0(\mf{m}_-, V_k(\fing)_{[0]})\cong \C$,
and therefore $H_{DS,f}^0(V_k(\fing))_{[0]}$  is nonzero.
Thus, $H_{DS,f}^0(V_k(\fing))$ is nonzero, and this completes the proof.
\end{proof}

Let $f\in \mathbb{O}_{(2^{r-2},1^4)}$.
By Theorem \ref{Th:W-algebra-variety} and Theorem \ref{Th:AMconj},
$\W_{2-r}(\mf{so}_{2r},f)$, for 
$r$ even, is lisse.
The central charge of  $\W_{k}(\mf{so}_{2r},f)$
is given by
$$-\frac{(k+r-2) \left(3 k r-6 k+2 r^2-12 r+10\right)}{k+2 r-2}.$$
In particular the central charge of $\W_{2-r}(\mf{so}_{2r},f)$ is zero. 
Note that the central charge of $\W_{k}(\g,f)$ is zero if  
$\W_{k}(\g,f)$ is trivial. 
\begin{Conj} \label{Conj:lisse-typeD}
Let $r$ be even, $r \geq 4$, and $f\in \mathbb{O}_{(2^{r-2},1^4)}$.
\begin{enumerate}
\item $\W_{2-r}(\mf{so}_{2r},f)=\C$.
\item $\W_{k}(\mf{so}_{2r},f)$ is lisse for any integer $k$ such that $k 
\geq 2-r$.
\end{enumerate}
\end{Conj}

Parts (1) and (2) of Conjecture \ref{Conj:lisse-typeD} 
are true for $r=4$ by \cite{AM15}.  
Also, Part (2) of the conjecture is true for $k=2-r$ as mentioned 
just above. 

\section{The short nilpotent orbit closure in type $B$}
\label{sec:type-B-short}
In this section, $\fing=\mf{so}_{2r+1}$ with $r\geq 3$.
Let $\Delta=\{\pm \varepsilon_i\pm  \varepsilon_j, \pm  \varepsilon_j 
\; |\; 1\leq i,j\leq {r }, i\not= j\}$ be the root system of $\g$ 
and take $\Delta_+=\{\varepsilon_i\pm  \varepsilon_j, \varepsilon_k  
\; |\; 1\leq i < j\leq {r }, 1 \leq k \leq {r } \}$ for the set of positive roots. 

Denote by $(e_i,h_i,f_i)$ the Chevalley generators of $\g$ 
 in the Bourbaki numbering, and 
fix the root vectors $e_\alpha,f_\alpha$, 
$\alpha \in \Delta_+$ so that $(h_i,\,i=1,\ldots,{r }) \cup 
(e_\alpha,f_\alpha, \, \alpha \in \Delta_+)$ is a Chevalley 
basis satisfying the conditions of \cite[Chapter IV, Definition 6]{Gar}. 
Denote by $\alpha_1,\ldots,\alpha_r$ the corresponding simple roots of $\g$. 
Let $\g= \n_- \oplus \h \oplus \n_+$ be the corresponding triangular decomposition. 
For $\alpha \in \Delta_+$, denote by $h_\alpha =[e_\alpha,f_{\alpha}]$ 
the corresponding coroot. 

%The fundamental weights are: 
%$$\varpi_i = \varepsilon_1+\cdots +\varepsilon_i \quad (1 \leq i < {r }), 
%\qquad \varpi_{r } = \frac{1}{2}( \varepsilon_1+\cdots +\varepsilon_{r }).$$
Denote by $\varpi_1,\ldots,\varpi_r$ the fundamental weights 
and  by $\varpi_1^\vee,\ldots,\varpi_r^\vee$ the corresponding 
fundamental co-weights. Note that $\varpi_i^\vee$ and 
$\kappa^\sharp(\varpi_i)$ are proportional, with $\kappa^\sharp \colon\h^* \to \h$ 
the isomorphism induced from $(~|~)$. 
In particular $\kappa^\sharp(\varpi_1)=\varpi_1^\vee$. 

There is a unique short nilpotent orbit
$\mathbb{O}_{short}$ in $\fing$,
which is the nilpotent orbit associated with 
the $\mf{sl}_2$-triple $(e_{\theta_s}, h_{\theta_s},f_{\theta_s})$,
 where $\theta_s$ is 
 the highest short root $\varepsilon_1$ 
 and where  $h_{\theta_s}=2\varpi_1^{\vee}$.
So, 
$$\fing=\fing_{-1}\+\fing_0\+\fing_1,$$
where $\fing_j=\{x\in \fing\mid [h_{\theta_s},x]=2jx\}$. 
The nilpotent orbit $\mathbb{O}_{short}$ is labeled by
 the partition
$(3,1^{2r-2})$, and its 
weighted Dynkin diagram is given by
$$\begin{Dynkin}
	\Dbloc{\Dcirc\Deast\Dtext{t}{2}}
	\Dbloc{\Dcirc\Dwest\Deast\Dtext{t}{0}}
	\Dbloc{\Dcirc\Dwest\Deast\Dtext{t}{0}}
	\Dbloc{\Dcirc\Dwest\Deast\Dtext{t}{0}}
	\Dbloc{\Ddots}
	\Dbloc{\Dcirc\Dwest\Ddoubleeast\Dtext{t}{0}}
	\Drightarrow
	\Dbloc{\Dcirc\Ddoublewest\Dtext{t}{0}}
\end{Dynkin}$$
We have
\begin{align*}
\fing^{\natural}:=\fing_0^{f_{\theta_s}}=\bra e_{\alpha},f_{\alpha}\mid
\alpha=\varepsilon_i-\varepsilon_j, \varepsilon_i+\varepsilon_j\; ; 
2\leq i<j\leq r \ket \cong \mf{so}_{2r-2}.
\end{align*}
In particular, $\fing^{\natural}$ is simple.
As a module over $\fing^{\natural}=\mf{so}_{2r-2}$,
$\fing_{-2}$ decomposes as $\fing_{-2}\cong \C\+ \C^{2r-2}.$ 

In this section we prove the following assertion.
\begin{Th}\label{Th:short}
Let $\fing=\mf{so}_{2r+1}$, with $r \geq 3$,  
and let $k=-2$.
Then 
$$X_{V_k(\fing)}=\overline{\mathbb{O}_{short}}.$$
\end{Th}
Let $w_2$ 
be the singular vector with respect to the adjoint action of $\g$ 
as in \cite[Theorem 4.2]{AM15} 
which generates the irreducible representation $W_2$   
of $\g$ with highest weight $\theta+\theta_2$ in $S^{2}(\g)$, where 
$S^{2}(\g)$ is the component of $S(\g)$ of degree $2$ 
and $\theta_2$ is the highest root of the simple Lie algebra 
generated by the roots $\alpha_3,\ldots,\alpha_{r }$.  
%%%% CHANGE FOR THE PUBLISHED VERSION%%%%
Namely,
$$w_2=e_{\theta} e_{\theta_2}-\sum_{j=1}^2 e_{\beta_j+\theta_2} 
e_{\delta_j+\theta_2},$$
with
$\beta_1 := \alpha_2$, $\delta_1 := \alpha_1+\alpha_2+\alpha_3$, 
$\beta_2 := \alpha_2+\alpha_3$, $\delta_2 := \alpha_1+\alpha_2.$
Then $\sigma(w_2)$ is a singular vector of $V^k(\g)$ if and 
only if $k=-2$  \cite[Theorem 4.2]{AM15} where 
$\sigma$ 
is the natural embedding of $\g$-modules 
from $S^2(\g)$ to 
$V^{k}(\g)_2:=\{v \in V^k(\g) \; |\; D v = - 2v\}$ 
(see \cite[Lemma 4.1]{AM15}). 

Denote by $V(W_2)$ the zero locus in $\g^*\cong\g$ 
of the ideal of $S(\g)$ generated by $W_2$. 

\begin{lemma} \label{lem2:B-zero}
We have $V(W_2) \cap \mc{N} \subset \overline{\O_{(3,1^{n-3})}}$.
\end{lemma}

\begin{proof}
We observe that that $(3,1^{2r-2}) \in \P_1(2r+1)$ is the unique minimal 
$1$-degeneration of $(3,2^2,1^{2r-6}) \in \P_1(2r+1)$ (see 
Definitions~\ref{Def:degeneration} and~\ref{Def:deg-min}). 
Therefore, it is enough to show that $V(W_2)$ do not contain 
$\overline{\O_{(3,2^2,1^{2r-6})}}$. 
To proceed, we apply \cite[Lemma 3.3]{AM16} 
and we argue as in the proofs 
of Lemma 7.6 and Lemma 9.2 of \cite{AM16}. 
Since the verification are similar 
we omit the details here. 
\end{proof}

Set $S(\g)^\h=\{x \in S(\g) \; | \; [h,x]=0 \text{ for all } h \in \h\}$ 
and let 
$\Psi \colon S(\g)^\h \to S(\h)$ 
be the Chevalley projection map.    

\begin{lemma} \label{lem:B-zw2} 
The zero weight space of $W_2$ has dimension ${r }({r }-1)/2$. 
\end{lemma}

\begin{proof} 
It is easy to verify the statement for $r=3$ and $r=4$ where $W_2$ 
has highest weight $2\varpi_{r }$. 
For ${r } \geq 5$, 
$W_2$ has highest weight $\varpi_4$. 
Hence $W_2 \cong \wedge^4 \C^{2r+1}$ and it is easy to show that the zero weight 
space of $W_2$ has dimension ${r }({r }-1)/2$. 
\end{proof}

Set $W_2^{\h,({r })}:=\Psi(W_2 \cap S(\g)^\h)$. 
By Lemma \ref{lem:B-zw2}, $W_2^{\h,({r })}$  
has dimension ${r }({r }-1)/2$. 

\begin{lemma} \label{lem:B-gen}
We can describe a set of generators of $W_2^{\h,({r })}$ by induction on ${r }$ as follows. 
Set 
\begin{align*} 
&& p_1^{(3)}  := h_1 h_3, \quad 
p_2^{(3)} :=  ( 2 h_2 + h_3) h_3 , \quad 
p_3^{(3)} :=  (h_1 + h_2) h_2 + (h_2 h_3)/2. 
\end{align*}
Assume that ${r } \geq 4$ and that generators $p_1^{(k)},\ldots,p_{k(k-1)/2}^{(k)}$ 
of $W_2^{\h,(k)}$ have been constructed for any $k \in\{3,\ldots,{r }-1\}$.   
Set: 
\begin{align*} 
&p_{j}^{({r })} :=  h_1 h_{j+2}, \; j=1,\ldots,{r }-2, \quad
p_{k}^{({r })} :=  p_k ^{({r }-1)}, \; k = r -1,\ldots, r(r-1)/2-1,\\ 
&p_{l}^{({r })} :=   (h_1 + h_2) h_2 + \tilde{p}_1^{({r })}, \; l=r(r-1)/2, 
\end{align*}
where $\tilde{p}_1^{({r })}$ is a homogeneous polynomial of degree 2 in the variables 
$h_2,\ldots,h_{r }$ with no term in $h_2^2$. 
Then $p_1^{({r })},\ldots,p_{{r }({r }-1)/2}^{({r })}$ generate $W_2^{\h,({r })}$. 
\end{lemma}

\begin{proof} 
%The proof is technical but elementary. 
%We omit the proof here. 
%%%% CHANGE FOR THE PUBLISHED VERSION%%%%
%
Since the $r(r-1)/2$ elements $p_i^{({r })}$ as described in the lemma 
are linearly independent, it suffices to prove that they are elements of 
$W_2^{\h,({r })}$ by Lemma \ref{lem:B-zw2}. 

First of all, we have the following relations whose verifications are left to the reader:
\begin{align*}
&[f_{\beta_1+\theta_2},[ f_{\delta_1+\theta_2}, w_2]] \equiv   h_{\delta_1+\theta_2} h_{\beta_1+\theta_2} ,\qquad  
[f_{\delta_2+\theta_2},[ f_{\beta_2+\theta_2}, w_2]] 
\equiv   h_{\delta_2+\theta_2} h_{\beta_2+\theta_2} , \\
&  [f_{\theta_2},[ f_{\theta}, w_2]] \equiv   h_{\theta} h_{\theta_2}, \qquad   [f_{\alpha_4},[f_{\alpha_3+\alpha_4 +2(\alpha_5 +\cdots+\alpha_{{r }})},[f_{\theta},w_2]]] 
\equiv  h_{\theta} h_{\alpha_4} , \\
& [f_{\alpha_3+\alpha_4},[f_{\alpha_4+2(\alpha_5 +\cdots+\alpha_{{r }})},[f_{\theta},w_2]]] 
\equiv  h_{\theta} h_{\alpha_3+\alpha_4} ,  
\end{align*}
where $\equiv$ means that the equality is modulo $(\mf{n}_- +\mf{n}_+) S(\g)$. 
Moreover, if ${r } \geq 5$ then for any $j =5,\ldots,{r }-1$: 
\begin{align*}
[f_{\alpha_4+\cdots+\alpha_j},[f_{\alpha_3+\cdots+\alpha_j
+2(\alpha_{j+1} +\cdots+\alpha_{{r }})},[f_{\theta},w_2]]] 
&\equiv h_{\theta} h_{\alpha_4+\cdots+\alpha_j} . &
\end{align*}
From this, we deduce that the elements 
$$h_{\delta_1+\theta_2} h_{\beta_1+\theta_2}, 
\quad h_{\delta_2+\theta_2} h_{\beta_2+\theta_2}, \quad  
h_{\theta} h_{\theta_2}, \quad h_{\theta} h_{\alpha_4}, 
\quad h_{\theta} h_{\alpha_3+\alpha_4}$$
are in $W_2^{\h,({r })}$. Moreover, if ${r } \geq 5$, then the element 
$h_{\theta} h_{\alpha_4+\cdots+\alpha_j}$ 
is in $W_2^{\h,({r })}$ for any $j =5,\ldots,{r }-1$. 
For ${r } \geq 5$, 
$h_{\theta_2} = h_3 + 2h_4 +\cdots 2 h_{{r }-1} +h_{r },$ 
and for ${r }=4$,  
$h_{\theta_2} = h_3 + h_4.$ 
Then, since the elements $h_{\theta} h_{\alpha_4+\cdots+\alpha_{{r }-1}}$, 
$h_{\theta} h_{\alpha_4+\cdots+\alpha_j} -h_{\theta} h_{\alpha_4+\cdots+\alpha_{j-1}}$  
for $j =5,\ldots,{r }-1$, and $h_{\theta} h_{\alpha_3+\alpha_4}-h_{\theta} h_{\alpha_4}$ 
are in $W_2^{\h,({r })}$, 
the elements $h_{\theta} h_{\alpha_j}$ also belong to $W_2^{\h,({r })}$  
for $j=3,\ldots,{r }$ and ${r } \geq 4$. 
In conclusion, the elements 
\begin{eqnarray} \label{eq:elem}
& &  \quad h_{\delta_1+\theta_2} h_{\beta_1+\theta_2}, \quad h_{\delta_2+\theta_2} h_{\beta_2+\theta_2}, 
\quad h_{\theta} h_{\alpha_j} = h_\theta h_j ,\quad  j =3,\ldots,{r }-1,
\end{eqnarray}
are in $W_2^{\h,({r })}$.

We now prove the statement by induction on ${r }$. 

\noindent
$\ast$ ${r }=3$. Since $q_1^{(3)} := h_{\delta_1+\theta_2} h_{\beta_1+\theta_2}$, 
$q_2^{(3)} := h_{\delta_2+\theta_2} h_{\beta_2+\theta_2}$ and 
$q_3^{(3)} := h_\theta h_{\theta_2}$ 
 are in $W_2^{\h,(3)}$, 
%\begin{align*} 
%q_1^{(3)} &:= h_{\delta_1+\theta_2} h_{\beta_1+\theta_2} = (h_1+h_2+h_3)(2 h_2 + h_3) ,& \\
%\quad q_2^{(3)} := h_{\delta_2+\theta_2} h_{\beta_2+\theta_2}, \quad 
%= (2 h_1+2 h_2 +h_3) (h_2 + h_3),&\\ 
%q_3^{(3)} := h_\theta h_{\theta_2} = (h_1+2 h_2 +h_3) h_3,  &
%\end{align*}
the elements $p_1^{(3)} = q_2^{(3)} - q_1^{(3)}  = h_1 h_3$, 
$ p_2^{(3)} = q_3^{(3)} - h_1 h_3 = (2 h_2 +h_3) h_3$ and 
$p_3^{(3)} =  (q_1^{(3)} - q_3^{(3)} )/2 = (h_1+ h_2) h_2 + (h_2 h_3)/2$ 
%the elements 
%\begin{align*} 
%&p_1^{(3)} := q_2^{(3)} - q_1^{(3)}  = h_1 h_3,  \qquad p_2^{(3)} := q_3^{(3)} - h_1 h_3 = (2 h_2 +h_3) h_3,  \\
%&p_3^{(3)} :=  (q_1^{(3)} - q_3^{(3)} )/2 := (h_1+ h_2) h_2 + (h_2 h_3)/2 
%\end{align*}
are in $W_2^{\h,(3)}$ too, whence the statement. 

\noindent
$\ast$ ${r }=4$. 
Let $\g'$ be the subalgebra of $\g$ of type $B_{3}$ generated by 
$\alpha_2,\alpha_3,\alpha_4$ and assume the statement true for $\g'$. 
We denote by $w'_2$ the elements of $S^2(\g')$ corresponding to $\g'$, 
and by $\theta',\theta'_2,\beta'_j,\delta'_j$, $j=1,2$, 
the roots corresponding to $\theta,\theta_2,\beta_j,\delta_j$, $j=1,2$, 
for $\g'$. 
A direct computation shows that  
\begin{eqnarray} \label{eq:ind1}
[ f_{\alpha_1+\alpha_2} , [f_{\alpha_3+\alpha_4},w_2]] = 
e_{\theta'} e_{\theta'_2}-\sum_{k=1}^2 e_{\beta'_j+\theta'_2} e_{\delta'_j+\theta'_2} =w'_2.
\end{eqnarray}
It follows from (\ref{eq:ind1}) that the elements 
$p_1^{(3)}  =h_2 h_4$, $p_2^{(3)} = (2 h_3 +h_4)h_4$ and  
$p_3^{(3)}= (h_2+h_3)h_3 +(h_3 h_4)/2$ 
%\begin{align*} 
%p_1^{(3)}  =h_2 h_4, \quad  
%p_2^{(3)} = (2 h_3 +h_4)h_4, \quad 
%p_3^{(3)}= (h_2+h_3)h_3 +(h_3 h_4)/2 
%\end{align*}
are in $W_2^{\h,(4)}$. 
Moreover, by (\ref{eq:elem}), since 
$q_1^{(4)}:= h_{\delta_1+\theta_2} h_{\beta_1+\theta_2}$, 
$q_2 ^{(4)}:= h_{\delta_2+\theta_2} h_{\beta_2+\theta_2}$, 
$q_3^{(4)} := h_\theta h_{\alpha_3}$ and 
$q_4^{(4)}:= h_\theta h_{\alpha_4}$ 
are in $W_2^{\h,(4)}$, 
%the following elements are in $W_2^{\h,(4)}$:
%\begin{align*} 
%q_1^{(4)}:= h_{\delta_1+\theta_2} h_{\beta_1+\theta_2} &= (h_1+h_2+2 h_3+h_4 )(h_2 + h_3 +h_4) , &\\
%q_2 ^{(4)}:= h_{\delta_2+\theta_2} h_{\beta_2+\theta_2} &= (h_1+h_2 +h_3+h_4) (h_2 + 2 h_3 +h_4), &\\ 
%q_3^{(4)} := h_\theta h_{\alpha_3} &= (h_1+2 h_2 +2 h_3 + h_4) h_3, &\\
%q_4^{(4)}:= h_\theta h_{\alpha_4} &= (h_1+2 h_2 +2 h_3 + h_4) h_4.& 
%\end{align*}
the elements $q_2^{(4)} - q_1^{(4)} = h_1 h_3$, 
$q_4^{(4)} - 2 h_2 h_4 - (2 h_3 + h_4)h_4 =  h_1 h_4$ and 
$q_1^{(4)} -   h_1 h_3 - h_1 h_4 = (h_1+h_2)h_2 + (h_3+h_4)(h_2+h_3+h_4)$ 
%\begin{align*} 
%q_2(^{(4)} - q_1^{(4)} = h_1 h_3 ,\quad q_4^{(4)} - 2 h_2 h_4 - (2 h_3 + h_4)h_4 =  h_1 h_4 ,\\
%q_1^{(4)} -   h_1 h_3 - h_1 h_4 = (h_1+h_2)h_2 + (h_3+h_4)(h_2+h_3+h_4) 
%\end{align*}
are in $W_2^{\h,(4)}$ too. 
Hence the statement is true for ${r }=4$. 

\noindent
$\ast$ Assume the statement true for any rank $< {r }$. 
Let $\g'$ be the subalgebra of $\g$ of type $B_{{r }-1}$ generated by 
$\alpha_2,\ldots,\alpha_{r }$ and assume the statement true for $\g'$. 
We denote by $w'_2$ the elements of $S^2(\g')$ corresponding to $w_2$ for $\g'$, 
and by $\theta',\theta'_2,\beta'_j,\delta'_j$, $j=1,2$, 
the roots corresponding to $\theta,\theta_2,\beta_j,\delta_j$, $j=1,2$, 
for $\g'$. 
%For example, $\theta' = \alpha_2+2\alpha_3 +\cdots+2\alpha_{r }$. 
A direct computation shows that  
\begin{eqnarray} \label{eq:ind}
[ f_{\alpha_1+\alpha_2} , [f_{\alpha_3+\alpha_4},w_2]] = 
e_{\theta'} e_{\theta'_2}-\sum_{k=1}^2 e_{\beta'_j+\theta'_2} e_{\delta'_j+\theta'_2} =w'_2.
\end{eqnarray}
From~(\ref{eq:ind}), we deduce that the elements $p_k^{({r }-1)}$, 
for $k =r-1,\ldots,{r }({r }-1)/2-1$, viewed as polynomials in the variables $h_2,\ldots,h_{r }$, 
are in $W_2^{\h,({r })}$. 
In addition, by~(\ref{eq:elem}), 
$q_1^{({r })} := h_{\delta_1+\theta_2} h_{\beta_1+\theta_2}$, 
$q_2^{({r })} := h_{\delta_2+\theta_2} h_{\beta_2+\theta_2}$, 
$q_3^{({r })}  := h_{\theta} h_{\theta_2}$ 
and $q_{j} ^{({r })} := h_\theta h_{\alpha_j}$, $j=3,\ldots,r $, 
are in $W_2^{\h,({r })}$. 
%\begin{align*} 
% q_1^{({r })} & := h_{\delta_1+\theta_2} h_{\beta_1+\theta_2} &\\
% &= (h_1+ h_2 +2 h_3+\cdots+ 2 h_{{r }-1}+ h_{r }) 
% (h_2 +h_3 +2 h_4+\cdots+2 h_{{r }-1}+ h_{r } ), &\\
% q_2^{({r })} & := h_{\delta_2+\theta_2} h_{\beta_2+\theta_2} & \\
% &= (h_1+h_2 +h_3 +2 h_4+\cdots+ 2 h_{{r }-1}+ h_{r }) 
% (h_2 + 2 h_3+\cdots+ 2 h_{{r }-1}+ h_{r }), &\\
%% q_3^{({r })} & = h_{\theta} h_{\theta_2} = (h_1 +2 h_2+\cdots+ 2 h_{{r }-1}+ h_{r }) 
% %(h_3 +2 h_4+\cdots+ 2h_{{r }-1}+ h_{r }), &\\
%q_{j} ^{({r })} & := h_\theta h_{\alpha_j} =   (h_1 +2 h_2+\cdots+ 2 h_{{r }-1}+ h_{r }) h_{j}, \quad   
%j \in \{3,\ldots,{r }\}. &
%\end{align*}
We first deduce that 
$q_2^{({r })} - q_1^{({r })} = h_1 h_3$ is in $W_2^{\h,({r })}.$
%Next, since 
%$$h_{\theta}h_{\alpha_4+\cdots+\alpha_{{r }-1}} 
%= h_{\theta}(h_4 + \cdots + h_{{r }-1} )\in W_2^{\h,({r })} ,$$
%we deduce using $q_3^{({r })}$ that 
%$$h_{\theta}(h_3 + \cdots + h_{{r }}) \in W_2^{\h,({r })},$$ 
%and hence using the $q_{j+1} ^{({r })}$'s for $j \in \{3,\ldots,{r }-1\}$, 
%that 
%$$q_{{r }+1} :=h_{\theta} h_{{r }}  \in W_2^{\h,({r })},$$
%too. 
Next, for $j=4,\ldots,r$, we have 
$$q_{j} ^{({r })} = 
h_1 h_j +h_2 h_j + (h_2 + 2 h_3 + \cdots + 2 h_{{r }-1} +h_{r }) h_j ,$$
whence $h_1 h_j \in W_2^{\h,({r })}$ by the induction hypothesis. 
%Indeed, $h_2 h_j$ and $(h_2 + 2 h_3 + \cdots + 2 h_{{r }-1} +h_{r }) h_j$ 
%belong to $W_2^{\h,{r }-1}$ by induction and the above discussion. 
Finally, we deduce that 
$q_1^{({r })}  - h_1(h_3 + 2 h_4+\cdots+2 h_{{r }-1}+ h_{r } )$ 
%\begin{align*}  
%&q_1^{({r })}  - h_1(h_3 + 2 h_4+\cdots+2 h_{{r }-1}+ h_{r } ) &\\ 
%&\qquad =  (h_1+h_2)h_2 
%+ (2 h_3+\cdots+ 2 h_{{r }-1}+ h_{r }) 
% (h_2 +h_3 + 2 h_4+\cdots+2 h_{{r }-1}+ h_{r } ) &\\ 
%& \qquad  \qquad + h_2(2 h_3+\cdots+2 h_{{r }-1}+ h_{r })  &
%  \end{align*}
is in $W_2^{\h,({r })}$ 
since $h_1 h_j \in W_2^{\h,({r })}$ for any $j=3,\ldots,r$. 
We have proven the expected statement for ${r }$. 
\end{proof}

\begin{lemma} \label{lem:B-zero} 
Let $\lam$ be a nonzero semisimple element of $\g^*\cong \g$ which belongs to 
$V(W_2)$.  
Then either $\lam \in G.\C \varpi_1$, or $\lam \in G.\C (-\varpi_{i-1} +\varpi_{i})$ 
for some $i=2,\ldots,{r }-1$, 
or $\lam \in G.\C(- \varpi_{{r }-1}+2\varpi_{r })$. 
%Then $z \in V(W_2)$ if and only if either 
%$z \in \C \varpi_1$, or $z \in \C (-\varpi_{i-1} +\varpi_{i})$ for some $i\in \{2,\ldots,{r }-1\}$, 
%or $z \in \C(- \varpi_{{r }-1}+2\varpi_{r })$. 
In particular, 
$\lam \in V(W_2)$ if and only if 
$\lam \in \bigcup_{i=1}^r \C \eps_i=G.\C \varepsilon_1.$
\end{lemma}

\begin{proof} 
We have 
$$V(W_2) \cap \h^*=\{\lam \in \h^* \; |\; p(\lam)=0 \text{ for all } 
p \in \Psi(W_2 \cap S(\g)^\h)\}.$$
Since $V(W_2)$ is $G$-invariant, it is enough to prove the lemma for nonzero 
elements $\lam \in V(W_2) \cap \h^*$. 

Let $\lam \in \h^*$ that we write as $\lam=\sum_{i=1}^{r } \lam_i 
\varpi_i$, $\lam_i\in\C$. 
It suffices to prove that if $\lam \in V(W_2)$, then $\lam$ is the union 
of the sets as described in the lemma. Indeed, it is easy to verify that, conversely, 
these sets all lie in $V(W_2)$. 
%Then it is easy to prove the statement by induction on ${r }$ 
%using Lemma~\ref{lem:B-gen}. 
%%%% CHANGE FOR THE PUBLISHED VERSION%%%%

We prove the statement by induction on ${r }$. 
Assume that $\lam \in V(W_2)$. 
By Lemma~\ref{lem:B-gen}, we get 
$$\lam_1 \lam_3 =0, \qquad ( 2 \lam_2 + \lam_3) \lam_3 =0, 
\qquad (\lam_1 + \lam_2) \lam_2+(\lam_2 \lam_3)/2 =0 .$$
So, if $\lam_3 =0$, then either $\lam_2=0$, or $\lam_1=-\lam_2$, and 
if $\lam_3 \not=0$, then $\lam_1=0$ and $\lam_3 = -  2 \lam_2$, 
whence the statement. 

Assume ${r } \geq 4$ and the statement true for any rank $<{r }$. 
We have to solve the systems 
of equations $p_i^{({r })} (z)=0$ for $i =1,\ldots,r(r-1)/2$. 
Let $\g'$ be as in the proof of Lemma \ref{lem:B-gen}. 
If $\lam_1=0$ then by induction applied to $\g'$, we get the statement. 
Otherwise, using Lemma \ref{lem:B-gen} and the 
equation $0=p_j^{({r })}(\lam)=\lam_1 \lam_{j+1}$, for $j=1,\ldots,{r }-2$, 
we deduce that $\lam_3=\cdots=\lam_{r }=0$. 
Therefore, from 
$$0=p_{r(r-1)/2}^{({r })}(\lam)=\lam_2(\lam_1+\lam_2)$$ 
we deduce that 
either $\lam_2=0$ or $\lam_1=-\lam_2$, 
whence the statement. 

The last assertion is straightforward using the description of the 
fundamental weights $\varpi_1,\ldots,\varpi_r$.  
\end{proof}

Denote by $\l$ the standard Levi subalgebra generated by 
$\alpha_2,\ldots,\alpha_{r }$. 
Then $\l \simeq \C\times \mf{so}_{2r-2}$ and its center 
is generated by the semisimple element $\varpi_1^\vee$. 
Let $$\SS_{short}:=\SS_{\l}$$
be the 
Dixmier sheet associated with $\l$ (cf.\ \cite{AM16}).
It is the
 the unique sheet containing the nilpotent orbit 
$\O_{short}$ of $\so_{2r+1}$. 
In fact, $$\O_{short} = {\rm Ind}_{\l}^\g(0),$$
where $ {\rm Ind}_{\l}^\g(0)$ is the induced nilpotent orbit 
of $\fing$ from $0$ in $\l$, 
and $\O_{short}$ cannot be induced in another way, see e.g.~\cite[Appendix]{MorYu16}. 
We have (cf.~e.g. \cite{AM16})
$$\overline{\SS_{short}}=\overline{G.(\C^*\kappa^\sharp(\varpi_1))}.$$

%By Lemma \ref{lem2:B-zero}, Lemma \ref{lem:B-zero}
%and \cite[Lemma 2.1]{AM16}, 
%we obtain the following result since the elements $\varepsilon_i$ are all conjugate 
%to $\varpi_1=\eps_1$ under the Weyl group of $(\g,\h)$.  
%
%\begin{Th} \label{Th:B-zero} 
%We have $V(W_2)=\overline{\SS_{\l_2}}$. 
%\end{Th}

\begin{Pro}\label{Pro:sheet}
Let $\fing$, $k$ be as in 
Theorem \ref{Th:short}.
Then $X_{V_k(\fing)}\subset \overline{\SS_{short}}$.
\end{Pro}
\begin{proof}
%Let $v$ be the singular vector of $V_{k}(\fing)$ at $k=-2$ of weight
%$(\theta+\theta_2)-2\delta$
%described in 
% \cite[Theorem 4.2]{AM15},
% where 
% $\theta_2$ is the highest root of the simple Lie algebra 
%generated by the roots $\alpha_3,\ldots,\alpha_{r }$. 
Let $\tilde{V}_k(\fing)$ be the quotient of $V^k(\fing)$ 
by the submodule generated by
$\sigma(w_2)$.
Then $X_{\tilde{V}_k(\fing)}$ is the 
zero locus of the $\ad \fing$-submodule of $S(\fing)$ 
generated by the image of $w_2$ in $R_{V_k(\fing)}=S(\fing)$, 
that is, $V(W_2)$. 
By Lemma \ref{lem2:B-zero}, Lemma \ref{lem:B-zero}
and \cite[Lemma 2.1 (2)]{AM16}, 
we obtain that $X_{\tilde{V}_k(\fing)}= \overline{\SS_{short}}$. 
%As in the same way as \cite[Proposition 7.8, Proposition 8.10 ]{AM16},
%we find that $X_{\tilde{V}_k(\fing)}= \overline{\SS_{short}}$.
This completes the proof
as $V_k(\fing)$ is a quotient of $\tilde{V}_k(\fing)$.
\end{proof}

\begin{Pro}\label{Pro:nonzero-and-simple}
Let $\fing$, $k$ be as in 
Theorem \ref{Th:short}.
Then 
$H_{f_{\theta_s}}^0(V_{k}(\fing))$ is nonzero and simple.
In particular,  $$X_{V_{k}(\fing)}\supset \overline{\O_{short}}.$$
\end{Pro}
\begin{proof}
Consider the subalgebra $\mf{a}$ of $\affg$ generated by
$x_{\alpha}\otimes t$, $\alpha\in \Delta_{-1}$,
$x_{\beta}$, $\beta\in \Delta_{0}$,
$x_{\gamma}\otimes t^{-1}$, $\gamma\in \Delta_{1}$,
which is isomorphic to $\fing$.
The vacuums vector $|0\ket$ generates the irreducible  representation of 
$\mf{a}=\g$ with highest weight
$-2\varpi_1$.
Then as in the proof of Theorem \ref{Th:AMconj}, 
we use Jantzen's simplicity criterion \cite{Jan77} (see also \cite[Section~9.13]{Humphreys}) 
to find that
$
U(\mf{a})|0\ket \cong U(\mf{a})\*_{U(\mf{p})}\C |0\ket,
$
where $\mf{p}$ is the parabolic subalgebra of $\mf{a}$ generated by 
$x_{\alpha}\otimes t$, $\alpha\in \Delta_{-1}$,
$x_{\beta}$, $\beta\in \Delta_{0}$.
It follows from
 \cite[Theorem 6.3]{AM16} that
$H_{f_{\theta_s}}^0(V_{k}(\fing))$ is nonzero and almost irreducible.
As $H_{f_{\theta_s}}^0(V_{k}(\fing))$ is cyclic,  the almost irreduciblity is the same as the simplicity.
Finally, the last assertion follows from Theorem \ref{Th:W-algebra-variety} (2).
\end{proof}

\begin{Lem}\label{Lem:or}
Let $Y$ be a $G$-invariant subvariety of $ \overline{\SS_{short}}$
containing $\overline{\O_{short}}$.
Then $Y= \overline{\SS_{short}}$ or $Y=\overline{\O_{short}}$.
\end{Lem}
\begin{proof}
By \cite[Lemma 2.1]{AM16},
$\overline{\SS_{short}}=G.\C^*\lam\cup \overline{\O_{short}}$.
\end{proof}

Let  
$\on{Vir}^c$ denote the universal Virasoro vertex algebra  at central charge $c\in \C$,
and $\on{Vir}_c$  the unique simple quotient of $\on{Vir}^c$.
We know that 
the following conditions are equivalent (\cite{BeiFeiMaz,Wan93}, see also \cite{Ara12,A2012Dec}):
\begin{enumerate}
\item  $\on{Vir}_c$ is rational,
\item  $\on{Vir}_c$ is lisse,
\item $c=c_{p,q}:=1-6(p-q)^2/pq$ with $p,q\in \Z_{\geq 2}$, $(p,q)=1$.
\end{enumerate}

\begin{Th}\label{Th:reduction-is-simple-virasoro}
Let $\fing$, $k$ be as in 
Theorem \ref{Th:short}.
We have
$H_{f_{\theta_s}}^0(V_{k}(\fing))\cong \on{Vir}_{c_{2,2r-3}} $.
Thus, $\W_{k}(\fing,f_{\theta_s})\cong  \on{Vir}_{c_{2,2r-3}}$.
In particular $\W_{k}(\fing,f_{\theta_s})$ is lisse
and rational.
\end{Th}
\begin{proof}
First it is straightforward to check that
$H_{f_{\theta_s}}^0(V_{k}(\fing))$ and $ \on{Vir}_{c_{2,2r-3}} $
has the same central charge.

Let $\tilde{V}_k(\fing)$  be as in the proof of Proposition \ref{Pro:sheet}.
Then 
$X_{H_{DS,f_{\theta_s}}^0(\tilde{V}_k(\fing))}= \overline{\SS_{short}}\cap \Slo_{f_{\theta_s}} $
by Theorem \ref{Th:W-algebra-variety}. 
%\eqref{:eq:var-of-reduction}.
Since $f_{\theta_s}+e_{\theta_s}$ is a semisimple element that is conjugate to 
$h_{\theta_s}$,
we have (see \cite[Remark 3.5]{AM16}), 
$$X_{H_{DS,f_{\theta_s}}^0(\tilde{V}_k(\fing))}=f_{\theta_s}+ \C e_{\theta_s}.$$
This implies that 
$R_{H_{DS,f_{\theta_s}}^0(\tilde{V}_k(\fing))}$ is the  polynomial ring 
generated by the image of the conformal vector $L$ of 
$H_{DS,f_{\theta_s}}^0(\tilde{V}_k(\fing))$.
Hence,
$H_{DS,f_{\theta_s}}^0(\tilde{V}_k(\fing))$ is strongly generated by $L$.
As $V_k(\fing)$
is a quotient of $\tilde{V}_k(\fing)$,
the exactness of  $H_{DS,f_{\theta_s}}^0(?)$
(Theorem~\ref{Th:W-algebra-variety} (1)) implies that
$H_{DS,f_{\theta_s}}^0({V}_k(\fing))$  is also strongly generated 
by its conformal vector.
Since it is simple 
by Proposition \ref{Pro:nonzero-and-simple}, 
$H_{DS,f_{\theta_s}}^0({V}_k(\fing))$  must be isomorphic to 
$\on{Vir}_{c_{2,2r-3}}$.
\end{proof}

\begin{proof}[Proof of Theorem \ref{Th:short}]
By Proposition \ref{Pro:sheet} and Proposition \ref{Pro:nonzero-and-simple},
$X_{V_k(\fing)}$ is
a $G$-invariant subvariety of $ \overline{\SS_{short}}$
containing $\overline{\O_{short}}$.
Thus
 $X_{V_k(\fing)}= \overline{\SS_{short}}$ or $X_{V_k(\fing)}=\overline{\O_{short}}$ 
by Lemma \ref{Lem:or}.
But if $X_{V_k(\fing)}= \overline{\SS_{short}}$ then
$X_{H_{DS,f_{\theta_s}}^0(\tilde{V}_k(\fing))}= \overline{\SS_{short}}\cap  \Slo_{f_{\theta_s}} $
is one-dimensional and $H_{DS,f_{\theta_s}}(V_k(\fing))$ cannot be lisse,
which contradicts Theorem~\ref{Th:reduction-is-simple-virasoro}, 
whence $X_{V_k(\fing)}=\overline{\O_{short}}$.
\end{proof}

\begin{Conj}\label{Conj:typeB-lisse}
Let $\fing=\mf{so}_{2r+1}$, $r \geq 3$.
Then $\W_k(\fing,f_{\theta_s})$ is lisse for any integer $k$ such that $k\geq -2$.
\end{Conj}

Conjecture \ref{Conj:typeB-lisse} is true for $k=-2$ by Theorem \ref{Th:short}. 

\newcommand{\etalchar}[1]{$^{#1}$}

%\bibliographystyle{alpha}

%\bibliography{/Users/tomoyuki/Documents/Dropbox/bib/math}

\end{document}